\documentclass[acmsmall]{acmart}

\usepackage{multicol}
\usepackage{booktabs}
\usepackage{bm}
\usepackage{bbm}
\usepackage{amsmath}
\usepackage{amsfonts}
\usepackage{subfigure}
\usepackage{pbox}
\usepackage{graphicx}
\usepackage{comment}

\usepackage[Symbolsmallscale]{upgreek}
\newcommand\todo[1]{\textcolor{red}{#1}}

\newcommand{\R}{\mathbb{R}}
\newcommand{\N}{\mathbb{N}}
\newcommand{\Z}{\mathbb{Z}}
\newcommand{\Char}{\mathbbm{1}}
\newcommand{\ones}{\mathbf{1}}
\newcommand{\M}{\mathbf{M}}
\newcommand{\mL}{\mathbf{L}}
\newcommand{\mW}{\mathbf{W}}
\newcommand{\mD}{\mathbf{D}}

\newcommand{\J}{\mathbf{J}}
\newcommand{\mQ}{\mathbf{Q}}
\newcommand{\mP}{\mathbf{P}}

\newcommand{\mA}{\mathbf{A}}
\newcommand{\mB}{\mathbf{B}}
\newcommand{\mC}{\mathbf{C}}
\newcommand{\mE}{\mathbf{E}}

\newcommand{\eye}{\mathbf{I}}
\newcommand{\vp}{\mathbf{p}}
\newcommand{\0}{\mathbf{0}}
\newcommand{\vd}{\mathbf{d}}
\newcommand{\vx}{\mathbf{x}}
\newcommand{\vy}{\mathbf{y}}
\newcommand{\va}{\mathbf{a}}
\newcommand{\vb}{\mathbf{b}}
\newcommand{\vz}{\mathbf{z}}
\newcommand{\ve}{\mathbf{e}}

\newcommand{\vecv}{\mathbf{v}}
\newcommand{\vw}{\mathbf{w}}
\newcommand{\vu}{\mathbf{u}}
\newcommand{\vf}{\mathbf{f}}

\newcommand{\cG}{\mathcal{G}}
\newcommand{\cN}{\mathcal{N}}
\newcommand{\cH}{\mathcal{H}}
\newcommand{\cE}{\mathcal{E}}

\newcommand{\cL}{\mathcal{L}}
\newcommand \innprod[2]{\langle {#1},{#2} \rangle_\frac{1}{\pi}}

\theoremstyle{acmplain}
\newtheorem{theorem}{Theorem}[section]

\newtheorem{proposition}[theorem]{Proposition}
\newtheorem{lemma}[theorem]{Lemma}

\newtheorem{definition}[theorem]{Definition}
\newtheorem{remark}[theorem]{Remark}

\allowdisplaybreaks

\def\BibTeX{{\rm B\kern-.05em{\sc i\kern-.025em b}\kern-.08emT\kern-.1667em\lower.7ex\hbox{E}\kern-.125emX}}

\begin{document}
\title{Fiedler Vector Approximation via Interacting Random Walks}

\author{Vishwaraj Doshi}
\affiliation{%
  \department{Operations Research Graduate Program}
  \institution{North Carolina State University}
  \city{Raleigh}
  \state{NC}
  \postcode{27695}
  \country{USA}}
\email{vdoshi@ncsu.edu}

\author{Do Young Eun}
\affiliation{%
  \department{Department of Electrical and Computer Engineering}
  \institution{North Carolina State University}
  \city{Raleigh}
  \state{NC}
  \postcode{27695}
  \country{USA}}
\email{dyeun@ncsu.edu}

\begin{abstract}
The Fiedler vector of a graph, namely the eigenvector corresponding to the second smallest eigenvalue of a graph Laplacian matrix, plays an important role in spectral graph theory with applications in problems such as graph bi-partitioning and envelope reduction. Algorithms designed to estimate this quantity usually rely on a priori  knowledge of the entire graph, and employ techniques such as graph sparsification and power iterations, which have obvious shortcomings in cases where the graph is unknown, or changing dynamically. In this paper, we develop a framework in which we construct a stochastic process based on a set of interacting random walks on a graph and show that a suitably scaled version of our stochastic process converges to the Fiedler vector for a sufficiently large number of walks. Like other techniques based on exploratory random walks and on-the-fly computations, such as Markov Chain Monte Carlo (MCMC), our algorithm overcomes challenges typically faced by power iteration based approaches. But, unlike any existing random walk based method such as MCMCs where the focus is on the leading eigenvector, our framework with interacting random walks converges to the Fiedler vector (second eigenvector). We also provide numerical results to confirm our theoretical findings on different graphs, and show that our algorithm performs well over a wide range of parameters and the number of random walks. Simulations results over time varying dynamic graphs are also provided to show the efficacy of our random walk based technique in such settings. As an important contribution, we extend our results and show that our framework is applicable for approximating not just the Fiedler vector of graph Laplacians, but also the second eigenvector of any time reversible Markov Chain kernel via interacting random walks. To the best of our knowledge, our attempt to approximate the second eigenvector of any time reversible Markov Chain using random walks is the first of its kind, opening up possibilities to achieving approximations of higher level eigenvectors using random walks on graphs.
\end{abstract}




\maketitle
\section{Introduction} \label{intro}
The eigenvector corresponding to the second smallest eigenvalue of a graph Laplacian matrix is usually referred to as the Fiedler vector. Originally introduced by Miroslav Fiedler in his works on algebraic connectivity \cite{Fiedler1,Fiedler2}, the Fiedler vector has found applications in areas such as graph partitioning and clustering \cite{Pothen1990,Pothen-sparse,min-max,Ng2001,slininger2013fiedlers,MulticlassSpectral,Dhillon,Orponen2005,pagerank-nibble}, graph drawing \cite{Drawing}, graph colouring \cite{GraphColoring}, envelope reduction \cite{Barnard1993-1} and the analysis of proteins \cite{Protein1,Protein2}. It plays an important role in spectral graph theory \cite{Mohar91, Chung1997}, providing powerful heuristics by solving relaxations of NP-hard, integer problems on graph bi-partitioning \cite{Tutorial, NCut, RCut}.

Over the years, a number of algorithms have been implemented to approximate the Fiedler vector. The most notable ones are techniques based on graph sparsification \cite{Spielman2011,Spielman2013,Spielman2014}, and multi-level/multi-grid techniques \cite{Barnard1993-1,Barnard1993-2, Urschel1, GraphContractionFV}. Both of these focus on pruning edges of the graph to obtain a `sparser' or `coarser' subgraph. Fiedler vector computation is then performed on these subgraphs to obtain approximations to the Fiedler vector of the original graph. While the pruning can be done via innovative probabilistic rules based on measures such as effective resistance \cite{Spielman2011}, or using greedy, deterministic rules \cite{Barnard1993-2}, the core computation of the Fiedler vector is still carried out by various kinds of power methods. The Fiedler vector is also directly related to the mixing time\cite{Levin,aldous,Bremaud} via the corresponding second eigenvalue $\lambda_2(\mQ)$, also known as the spectral gap, and algorithms to approximate this quantity \cite{Combes2019, I.Han-spectral,NIPS'15-spectral} are fundamentally different from those estimating the Fiedler vector.
Other techniques include \cite{BERTRAND2013}, which approximates the Fiedler vector in a distributed fashion over ad-hoc networks, where `ad-hoc' refers to the nodes having the ability to process information locally and exchange information with neighbors in a synchronous fashion, and `distributed' refers to each node estimating it's own component of the Fiedler vector using a local version of power iterations; and  \cite{Srinivasan2002}, which uses techniques such as matrix deflation to develop another power method to numerically compute the Fiedler vector using the dominant eigenvector of a slightly smaller matrix.

While  deterministic power method based techniques for approximating the Fiedler vector can have their advantages under the relevant settings (entire state space known beforehand, ad hoc networks with computational capability at each node and message passing), they face challenges when the state space may be unknown in the beginning, or the graph can only be explored via edge traversal mechanisms and direct access to any arbitrary node is not available. Moreover, they do not adapt well to dynamic graphs. This can especially be seen in \cite{BERTRAND2013} where recurring small changes in the graph topology can break the important mean-preserving property of their power method, which is corrected via a mechanism only at every $N^{th}$ iteration ($N$ being the number of nodes in the graph) of the algorithm. In other words, given a change in graph topology, the algorithm may not correct its trajectory till almost $N$ many steps, making the case worse for larger graphs. 
Random walk based methods, on the other hand, provide a way to deal with the above challenges by performing \emph{in situ} computations as they explore the graph on-the-fly. As exhibited time and again in the field of Markov Chain Monte Carlo (MCMC) \cite{Hastings1997,Jun_MCMC_Book,Peskun1973}, they can be used robustly to estimate target quantities on graphs without really feeling the repercussions of scale, lack of knowledge of the state space or the effect of dynamically changing graph topology. While MCMC techniques, by employing various versions of random walks and via the ergodic theorem, are successful in estimating $\pi$ (or sampling according to $\pi$) - the first/leading/principal eigenvector of the kernel $\mQ$, no similar techniques have provided extensions to the second eigenvector of the kernel.

In this paper, we fill this void by developing a framework based on (interacting) random walks to approximate the Fiedler vector of graph Laplacian matrices. We do this by constructing a stochastic process employing multiple \emph{interacting} random walkers and showing that a properly scaled version of this process converges to the Fiedler vector. Specifically, these random walkers traverse an undirected, connected graph $\cG$ according to a Continuous Time Markov Chain (CTMC) with kernel given by $\mQ = -\mL$, where the matrix $\mL \triangleq \mD - \mA$ is the combinatorial Laplacian of $\cG$ ($\mA$ is the adjacency matrix and $\mD$ is the degree matrix). Walkers are divided equally into two groups, that compete with each other over the network. If a walker encounters another one from the other competing group, it causes the walker to relocate to the location of another randomly selected walker of the competing group. Such competitive interactions are mutual, and for a sufficiently large number of walkers, lead to a natural bi-partition of the graph over time. By analyzing a closely related deterministic process (the fluid limit) we show that the relative density of walkers over the graph serves as a good approximation of the Fiedler vector.  We then extend our results to show that our method based on interacting random walks applies to other commonly used graph Laplacians, as well as for estimating the second eigenvector of any time reversible Markov chain kernel. While algorithms based on random walks successfully achieve, via ergodic theorems, knowledge about the \emph{first eigenvector} of matrics, our paper takes a step forward and provides an interacting random walk based algorithm to achieve the \emph{second eigenvector} of a class of matrices. To the best of our knowledge, our framework is the first one to do so, and opens up possibilities of estimating higher order eigenvectors by using such interacting random walk based techniques.

In the remainder of the paper, we begin by giving the basic notations and an introduction to the Fiedler vector via its application in graph partitioning in Section \ref{preliminaries}. The main theoretical results are distributed among Sections \ref{stochastic section}, \ref{deterministic section} and \ref{deterministic-stochastic}. In Section \ref{stochastic section}, we detail the construction of our stochastic process and via Theorem \ref{fluid limit} we relate it to a system of ordinary differential equations (ODEs) as its fluid limit. We show that over a finite time horizon, the stochastic process rarely deviates from the solution of the ODE when the number of walkers is sufficiently large. Section \ref{deterministic section} is devoted to the stability analysis of the resulting deterministic ODE system, where we show using a Lyapunov function that the Fiedler vector is the only asymptotically stable fixed point of a suitably scaled version of the system, while all others being unstable. In Section \ref{deterministic-stochastic} we bring together our results from Sections \ref{stochastic section} and \ref{deterministic-stochastic} to formally show that for sufficiently large number of walkers, our stochastic process spends most of its time in the long run around the asymptotically stable Fiedler vector, while never getting stuck around an unstable fixed point.  In Section \ref{Numericals}, we provide numerical results to support our theoretical findings and also simulations over time varying dynamic graphs to show the robustness of our framework in that setting. In Section \ref{time reversible}, we extend all our results to include various graph Laplacians and time reversible Markov chain kernels. Section \ref{conclusion} provides our concluding remarks.

\section{Preliminaries} \label{preliminaries}
\subsection{Basic notations} \label{notations}
Let $\cG(\cN,\cE)$ denote a general, undirected, connected graph, with $\cN$ being the set of nodes and $\cE$ being the set of edges, represented by pairs $(i,j)$ for $i,j\in\cN$. Let the cardinality of $\cN$ be given by a natural number $N$ (i.e. $|\cN| = N$). The mathematical quantities best capturing all the information of $\cG(\cN,\cE)$ are the `adjacency matrix' $\mA$\footnote{In the case of a weighted graph,  we replace $\mA$ by the weighted adjacency matrix $\mW$, where the $ij^{th}$ entries represent the weights assigned to each edge. All other equations remain the same. However, we will safely exclude any further, separate mention of weighted graphs because the scope of our results is broad enough to cover not just weighted graphs, but also similar quantities such as kernels of time reversible Markov chains, as we shall observe later in Section \ref{time reversible}.}  defined as
$\mA_{ij} \triangleq 1, \text{ if } (i,j) \in \cE$ and $0, \text{ otherwise}\ \forall i,j\in\cN,$
and the diagonal `degree matrix' $\mD$ of the graph defined as
$
\mD_{ii} \triangleq \sum_{j\in\cN} A_{ij} \ \ \forall i \in \cN.
$
We call $\mD_{ii}$ the `degree' of node $i\in\cN$, and alternatively represent it as $d(i) \triangleq \mD_{ii}$.

Since vectors and matrices will be used throughout the paper, we standardize their notation. Lower case, bold faced letters will be used to represent vectors (e.g. $\vecv \in \R^N$), while upper case, bold faced letters will be used to represent matrices (e.g. $\M \in \R^{N \times N}$), unless clarified otherwise. The $i^\text{th}$ ($ij^\text{th}$) entry of vector $\vecv$ (matrix $\M$) will be denoted by $v_i$ or $[\vecv]_i$ ($M_{ij}$ or $[\M]_{ij}$), depending on the situation. We let $\mD_\vecv := \text{diag}(\vecv)$ represent the diagonal matrix with $[\mD_\vecv]_{ii} = v_i$. Also denote by $\ones = [1, 1, \cdots, 1]^T$ and $\0 = [0, 0, \cdots, 0]^T$, the $N$-dimensional vectors of all ones and zeros respectively, and $\ve_k$ for all $k\in \{1,\cdots, N\}$ represent the canonical basis vectors in $R^N$, which take value $1$ at their $k^\text{th}$ entry, $0$ at every other entry.

The use of `$\mQ$' will be reserved exclusively for representing transition rate matrices of continuous time Markov chains (CTMCs). `$\mP(\cdot)$' will be used to denote the probability measure, while simple `$\mP$' will be used exclusively for transition probability matrices of discrete time Markov chains (DTMCs). Note that throughout the paper, we shall often refer to these matrices using the umbrella term \textit{kernel} of Markov chains. Vectors such as $\vx(t)$ and $\vy(t)$, with $t$ being time, will be used to denote (deterministic) solutions to systems of ordinary differential equations (ODEs). However, if indexed by a parameter, for example $\vx^n(t)$ and $\vy^n(t)$, they will be used to denote stochastic processes (also parameterized by $n$). This distinction between notations of similar deterministic and stochastic quantities shall be reiterated when we define such quantities later on.

Finally, we let $\| \cdot \|$ denote the Euclidean norm for any Euclidean space (i.e. with the appropriate dimensions implicitly understood), and use $\tilde{\cdot}$ to denote normalized versions of vectors, or sets containing normalized vectors.
\subsection{Graph Laplacians} \label{laplacians pre}
Consider an undirected, connected graph $\cG(\cN,\cE)$. The quantity of interest throughout the paper will be the \textit{Combinatorial Laplacian}, which is defined as $$\mL \triangleq \mD - \mA.$$
$\mL$ is a symmetric (due to the graph being undirected), positive-semidefinite matrix with non negative eigenvalues. Since the graph is connected, $\mA$, and as a result $\mL$, are irreducible matrices. Thus, the Perron Frobenius (PF) \cite{Horn, Carl2000} theorem applies and the smallest eigenvalue $0$ (with eigenvector being $\ones$) has multiplicity $1$. We denote by $0 = \lambda_1 < \lambda_2 \leq \cdots \leq \lambda_N$, the spectrum of $\mL$, with $\ones = \vecv_1,\ \vecv_2, \cdots ,\ \vecv_N$ being the corresponding eigenvectors.

Closely related to the combinatorial Laplacian is the \textit{normalized or symmetric Laplacian}, which goes by
$$ \cL \triangleq \mD^{-1/2}\mL\mD^{-1/2} = \eye - \mD^{-1/2}\mA\mD^{-1/2}.$$
Also a symmetric positive semi-definite matrix, it's smallest eigenvalue is $0$ (with corresponding eigenvalue being $\mD^{1/2}\ones$). Denote by $0 = \bar\lambda_1 < \bar \lambda_2 \leq \cdots \leq \bar \lambda_N$, the spectrum of $\mL$, with $\mD^{1/2}\ones = \bar\vecv_1,\ \bar\vecv_2, \cdots ,\ \bar\vecv_N$ being the corresponding eigenvectors.

The third graph Laplacian we introduce is the \textit{random walk Laplacian}, denoted by $\mL^{rw}$. We shall define it using the combinatorial and symmetric Laplacians as $$\mL^{rw} \triangleq \mD^{-1}\mL = \mD^{-1/2}\cL\mD^{1/2} = \eye - \mD^{-1}\mA.$$ 
It gets its name because $\mP = \mD^{-1}A$ is a stochastic matrix (all the rows add up to one) which defines a simple random walk in discrete time on the graph. $\mL^{rw}$ is not a symmetric matrix. However, note that due to the similarity transformation (second equality), it shares the same spectrum as $\cL$with its left eigenvectors given by $\vecv_i^{rw} = \mD^{1/2}  \bar \vecv_i$ where $\bar \vecv_i$ is an eigenvector of $\cL$. The PF eigenvector of $\mL^{rw}$ is therefore given by $\vecv_1^{rw} = \mD\ones$, the vector with its entries being the degree of the respective node.

\subsection{Graph bi-partitioning and Fiedler vector}
The second eigenvalue of $\mL$, i.e. $\lambda_2$, is known as the algebraic connectivity of a graph, and the corresponding eigenvector is commonly referred to as the \emph{Fiedler vector}. We shall however use the term \emph{Fiedler vector} to refer to second eigenvectors of any graph Laplacian matrix. To convey the relevance of the Fiedler vector, we look at the problem of bi-partitioning a graph, whose objective is to partition a connected graph into two connected subgraphs in the most `natural' way possible. This is often interpreted as partitioning the graph into two subgraphs $S$ and $S^c$ with the fewest number of edges interlinking them, but also not having meaningless solutions such as partitioning the graph at leaf nodes. This train of thought has evolved into the study of problems such as the \textit{Ratio-cut} problem (\textit{RCut}) \cite{Tutorial, RCut} and the \textit{Normalized-cut} problem (\textit{NCut}) \cite{Tutorial, NCut}, among others. Define $\text{Cut}(S) \triangleq (1/2)\sum_{i \in S} \sum_{j \in S^c} A_{ij},$ for any $S\subset \cN$. The \textit{RCut} problem for graph bi-partitioning is given by 
%----------
\begin{equation}
\text{RCut}(\cG) = \min_{S \subset \cN} \text{RCut}(S) = \min_{S \subset \cN} \Big( \frac{\text{Cut}(S)}{|S|} + \frac{\text{Cut}(S^c)}{|S^c|}  \Big),
\end{equation}
%----------
Unfortunately, minimizing the \textit{RCut} over all subsets of $\cN$ is an NP hard problem. \cite{RCut} showed that by enlarging the integer valued domain into the real valued domain, approximate solutions can be found efficiently. For the \textit{RCut} problem, this real valued relaxation is given by
$$
\text{RCut}(\cG) \approx \min_{\vf \perp \ones, \ \vf \neq 0}  \frac{\vf^T \mL \vf}{\vf^T\vf},
$$
the solution to which is well known, and is the Fiedler vector $\vecv_2$ of $\mL$. Likewise, a relaxation of the similar \textit{NCut}\footnote{The \textit{NCut} problem is similar to the \textit{RCut} problem with the terms $|S|$ and $|S^c|$ in the denominator replaced by $\text{Vol}(S)$ and $\text{Vol}(S^c)$, where $\text{Vol}(S) \triangleq \sum_{i \in S} d(i)$.} 
problem is solved by $\bar \vecv_2$, the Fiedler vector of the normalized Laplacian $\cL$ (or equivalently, the Fiedler vector of $\mL^{rw}$, since they share the same signs for the entries, leading to the same partition). Partitioning according to the signed entries of these second eigenvectors therefore solves natural relaxations to well-defined but difficult to solve bi-partitioning problems, and these observations have paved the way for spectral clustering\cite{Mohar91, Chung1997,Ng2001, NMF_clustering, MulticlassSpectral, Dhillon, RCut, NCut, Pothen1990,Simon1991} as a powerful data analysis tool.

\section{A multi-walk, interacting stochastic process with a deterministic limit} \label{stochastic section}
We aim to construct a random walk based process which, in the long run, can approximate the Fiedler vector (second eigenvector of $\mL=\mD-\mA$). Our process consists of multiple random walks interacting with each other in a prescribed manner. This can also be written as a density dependent process (with parameter `$n$' proportional to the number of random walks) that has a deterministic fluid limit. In this section we provide construction of our stochastic process, and the result on its convergence to a deterministic process (its fluid limit).

\subsection{The interacting stochastic process} \label{construction}
Let $\cG(\cN,\cE)$ be any undirected, connected graph, with $N=|\cN|$ as before. For any $n \in \N$, consider $2n$-many random walkers traversing $\cG(\cN,\cE)$ according to a CTMC generated by the kernel $\mQ = -\mL = \mA-\mD$. These $2n$ walkers are split into two groups, of $n$ walkers each, by labeling them as either `type-\emph{x}' or `type-\emph{y}'. Define the stochastic process $\big(X(t),Y(t)\big)_{t\geq0}$, where $X(t) \in \N_0^N$ ($Y(t) \in \N_0^N$) is a vector such that $X(t)_i$ ($Y(t)_i$) indicates the number of type-\emph{x} ( type-\emph{y}) walkers present at node $i\in\cN$ at time $t\geq0$. Since there are $n$ walkers of each type, we have $\sum_{i \in \cN} X(t)_i = \sum_{i \in \cN} Y(t)_i = n$ for all $t \geq 0$. The state space is given by
\begin{equation}
\Gamma_n \triangleq \left\lbrace (X,Y) \in \N_0^N \times \N_0^N\ \ \Bigg| \ \ \sum_{i \in \cN} X_i = \sum_{i \in \cN} Y_i = n \right\rbrace.
\label{bar S_n}
\end{equation}

Ever so often, a type-\emph{x} walker traversing the graph may find itself at a node with type-\emph{y} walkers present. Such an event is what we call an `interaction'. When an interaction occurs, every type-\emph{y} walker present `kills' each type-\emph{x} walker present with rate $\kappa/n$, for some scalar $\kappa \in (0,+\infty)$. In return, every type-\emph{x} walker present `kills' each type-\emph{y} walker present with rate $\kappa/n$. Another way to describe this interaction is in a pairwise manner. Given any possible `pair' of one type-\emph{x} and one type-\emph{y} walker present at node $j$, they both kill each other with rate $\kappa/n$. At every node $j \in \cN$ there are $X(t)_jY(t)_j$ such pairs at any time $t\geq0$. Thus, overall, type-\emph{x} or type-\emph{y} walkers die with rate $\frac{\kappa}{n}X(t)_jY(t)_j$. Upon being killed, the `dead' walker relocates to the position of another randomly selected walker of the same type. At time $t>0$, this is equivalent to saying that a killed type-\emph{x} walker will redistribute to some node $i \in \cN$ with probability $X^n(t)_i / n$, which is proportional to the number of type-\emph{x} walkers present at node $i$. Similarly, a killed type-\emph{y} walker will redistribute to node $i \in \cN$ with probability $Y^n(t)_i / n$.

\begin{figure*}
  \centering
  \includegraphics[width=\textwidth]{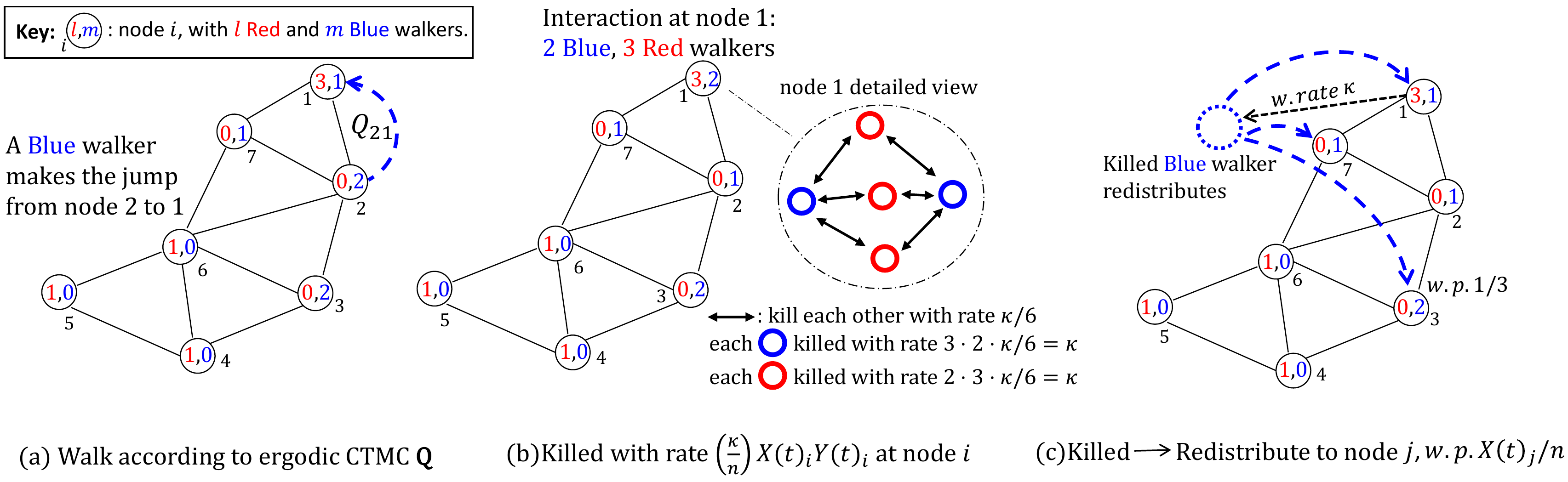}
  \caption{The interaction and redistribution mechanism.}
  \vspace{-5mm}
  \label{fig}
\end{figure*}

Consider Figure \ref{fig} as an example of an interaction event. We consider $n=6$. The colour blue is used for type-\emph{x} walkers, while red is used for type-\emph{y} walkers. At each node, the number of \emph{blue} and \emph{red} walkers present is given by the respectively coloured numeric entry. Figure \ref{fig}(a) shows a \emph{blue} walker moving from its initial position at node $2$, to node $1$, with rate $Q_{21}$. Figure \ref{fig}(b) shows a snapshot of the system at a time right after the event in Figure \ref{fig}(a). Node $1$ is now the site of interactions between $2$ \emph{blue} and $3$ \emph{red} walkers. A detailed view of the interactions at node $1$ is shown, where an arrow represents a pair of walkers killing each other with rate $\kappa/n$. Let $r \in \N$ count the number of arrows connected to any particular walker. Then, that walker is killed by the group of the other type walkers with the total rate $\frac{\kappa}{n}\cdot r$.
Following this logic, each \emph{blue} walker in Figure \ref{fig}(b) is killed with rate $\frac{\kappa}{6}\cdot 3 = \kappa/2$. Since there are two \emph{blue} walkers present, one of them is killed with rate $\frac{\kappa}{2}\cdot 2 = \kappa$. Similarly, each \emph{red} walker is killed individually with rate $\kappa/3$, and since there are three \emph{red} walkers, one of them is killed with rate $\kappa$. Therefore the rate with which a \emph{blue} walker dies is the same at which a \emph{red} walker dies at node $1$. Finally, Figure \ref{fig}(c) shows an event in which a \emph{blue} walker is killed at node $1$. Upon death, it instantaneously relocates to the position of another \emph{blue} walker chosen uniformly at random. That is, it randomly redistributes to node $i$ with probability $X(t)_j/n$. A similar redistribution would occur if a \emph{red} walker was to die instead of a \emph{blue} walker, in which case it would redistribute to node $j$ with probability $Y(t)_j/n$.

Events of the process $\big(X(t),Y(t)\big)_{t\geq0}$ involve type-\emph{x} or type-\emph{y} walkers going from a node $j$ to another node $i$. This corresponds to jumps of size $(\ve_i - \ve_j , 0)$ for type-\emph{x} and $(0,\ve_i - \ve_j)$ for type-\emph{y} walkers. At any state $(X,Y) \in \Gamma_n$, jumps can be caused by walking from $j \to i$ according to the base CTMC kernel $\mQ$, which happens with rate $Q_{ji}X_j$ for type-\emph{x} walkers and $Q_{ji}Y_j$ for type-\emph{y} walkers. The jumps can also occur as a result of being killed at node $j$ and redistributed to node $i$. As mentioned earlier, one of the type-\emph{x} walkers at node $j$ is killed with rate $\frac{\kappa}{n}X(t)_jY(t)_j$, upon which it redistributes to $i$ with probability $\frac{X_i}{n}$. Therefore, the jump $j \to i$ occurs with the overall rate $\frac{\kappa}{n}X(t)_jY(t)_j  \frac{X_i}{n}$ for type-\emph{x} walkers. Similarly, type-\emph{y} are killed at $j$ and redistributed to $i$ with overall rate $\frac{\kappa}{n}X(t)_jY(t)_j  \frac{Y_i}{n}$.

Hereafter, for any state $(X,Y) \in \Gamma_n$, we use $\vx \in \R^N$ to denote the density distribution of type-\emph{x} walkers over the graph, with $x_i \triangleq \frac{X_i}{N}$. Similarly we use $\vy \in \R^N$ to denote the density distribution of type-\emph{y} walkers over the graph, with $y_i \triangleq \frac{Y_i}{N}$. Let $\bar \mQ^n_{x:j \to i}(X,Y)$ denote the total rate with which jumps of type $(\ve_i - \ve_j , 0)$ occur, and similarly let $\bar \mQ^n_{y:j \to i}(X,Y)$ denote the total rate with which jumps of size $(0 , \ve_i - \ve_j )$ occur. From all of the above, these two quantities can be written as
\begin{align}
\bar Q^n_{x:j \to i} (X,Y) = Q_{ji}X_j + (\frac{\kappa}{n} Y_j X_j) x_i = n\big(Q_{ji}x_j + (\kappa y_j x_j) x_i\big)
\label{rate x} \\
\bar Q^n_{y:j \to i} (X,Y) = Q_{ji}Y_j + (\frac{\kappa}{n} X_j Y_j) y_i = n\big(Q_{ji}y_j + (\kappa x_j y_j) y_i\big)
\label{rate y}.
\end{align}
We can think of $\bar Q^n_{x:j \to i}$ and $\bar Q^n_{y:j \to i}$ for all $i,j \in \cN, i\neq j$ as the off-diagonal entries of a $2N \times 2N$ dimensional matrix $\bar\mQ(X,Y)$. This matrix is then the transition rate matrix of the CTMC $\big(X(t),Y(t)\big)_{t \geq 0}$ on a finite state space $\Gamma_n$ as in \eqref{bar S_n}, and we have now fully characterized our stochastic process.

%--------------------------------------------------------%
\subsection{A closely related deterministic system} \label{ODE intro}
Define the density dependent version of  $\big(X(t)/n,Y(t)/n \big)_{t\geq0}$ as
$\big(\vx^n(t),\vy^n(t)\big)_{t\geq0} \triangleq \big(X(t)/n,Y(t)/n \big)_{t\geq0}$. Its state space $\Theta_n$ is a version of $\Gamma_n$ where each entry is scaled by $\frac{1}{n}$, or more precisely,
\begin{equation}
\Theta_n = \Sigma_n \times \Sigma_n,
\label{Theta}
\end{equation}
where $\Sigma_n = \Sigma \cap \{\frac{1}{n}\vu \ | \ \vu \in \Z^N \}$, with $\Sigma \triangleq \{ \vu \in \R^N \ | \ \vu^T \ones = 1, \vu \geq 0\}$ and $\Z^N$ being the $N$-dimensional grid. Jumps of size $\frac{1}{n}(\ve_i - \ve_j, 0)$ and $\frac{1}{n}(0, \ve_i - \ve_j)$ for the process $\big(\vx^n(t),\vy^n(t)\big)_{t\geq0}$ occur with rates \eqref{rate x},\eqref{rate y}. Using this information, we can define a vector field $F:\Theta_n\rightarrow \R^{2N}$. For any $(\vx,\vy) \in \Theta_n$, $F(\vx,\vy)$ captures the average change in $(\vx,\vy) \in \Theta_n$ per unit time, and is written as
\begin{equation}
F(\vx,\vy) = \big(F_x(\vx,\vy), F_y(\vx, \vy)\big),
\label{Deterministic F(x,y)}
\end{equation}
where
\begin{align}
F_x(\vx,\vy) &\triangleq \sum_{i \in \cN} \sum_{j \in \cN, j \neq i} \frac{(\ve_i-\ve_j)}{n} \bar Q^n_{x:j,i} (\vx,\vy)
\label{mean field X}\\
F_y(\vx,\vy) &\triangleq \sum_{i \in \cN} \sum_{j \in \cN, j \neq i} \frac{(\ve_i-\ve_j)}{n} \bar Q^n_{y:j,i} (\vx,\vy).
\label{mean field Y}
\end{align}
The following result helps us write the above two equations in a more compact form.
%-------------------------------------------------
\begin{proposition}(Proof in Appendix \ref{proof 3.1}.) For any $(\vx,\vy) \in \Theta_n$, we have
\begin{align}
F_x(\vx,\vy) = \mQ^T\vx + [\kappa \vx^T\vy] \vx - \kappa \mD_{\vy} \vx
\label{F_x}\\
F_y(\vx,\vy) = \mQ^T\vy + [\kappa \vx^T\vy] \vy - \kappa \mD_{\vx} \vy
\label{F_y}
\end{align}
\label{Prop Simplify}
\end{proposition}%--------------------------------------------
Proposition \ref{Prop Simplify} allows us to consider the following system of ordinary differential equations (ODEs):
\begin{equation}
\begin{split}
\frac{d}{dt}\vx(t) = F_x\big(\vx(t),\vy(t)\big) \\
\frac{d}{dt}\vy(t) = F_y\big(\vx(t),\vy(t)\big) \\
\vx(0)^T\ones = \vy(0)^T\ones = 1,
\end{split}
\label{ODE F}
\end{equation}
where we let $\big(\vx(t),\vy(t)\big)_{t\geq 0}$ denote solutions to \eqref{ODE F} for $t \geq 0$ (this is sometimes referred to as the semi-flow of $F$). Note the distinction between the notations $\big( \vx^n(t),\vy^n(t) \big)$ and $\big(\vx(t),\vy(t)\big)$, where the former is a stochastic process (parameterized by $n$), while the latter is the deterministic solution to an ODE system. For the next result, let $\| \cdot \|$ denote the Euclidean norm in $\R^{2N}$, and $S$ be the subset of $\R^{2N}$ given by
\begin{equation}
S \triangleq \{(\vx,\vy) \in \R^{2N} ~|~ \vx^T\ones = \vy^T\ones = 1\}.
\end{equation}
%----------------------------------
\begin{proposition}(Proof in Appendix \ref{proof 3.2}.) $F:S\rightarrow\R^{2N}$ is Lipschitz continuous with Lipschitz constant $M<\infty$, and for every initial point $\big(\vx(0),\vy(0)\big) \in S$, there exists a unique solution to the system \eqref{ODE F}.
\label{Lipschitz prop}
\end{proposition}
%--------------------------------------------
The set $S$ is an invariant set for the ODE system \eqref{ODE F}. Indeed, observe that for any $(\vx,\vy \in S)$ we have $d(\ones^T\vx)/dt = \ones^T F(\vx,\vy) = \ones^T \mQ^T \vx + \kappa \vx^T \vy - \kappa\vy^T\vx = 0$ since $\mQ\ones = 0$. Similarly, $d(\ones^T\vy)/dt = 0$, and any solution starting in $S$ stays in $S$. We will only consider such solutions in our analysis of the ODE. To ensure that this is the case, we rewrite \eqref{ODE F} as an equivalent system given by
\begin{equation}
\begin{split}
\frac{d}{dt} \vx(t) = \mQ^T\vx(t) + \Lambda_x(t)\vx(t) - \kappa \mD_{\vy(t)} \vx(t) \\
\frac{d}{dt} \vy(t) = \mQ^T\vy(t) + \Lambda_y(t)\vy(t) - \kappa \mD_{\vx(t)} \vy(t) \\
\vx(t)^T\ones = \vy(t)^T \ones = 1 \ \big( \forall t\geq 0 \big),
\end{split}
\label{ODE XY}
\end{equation}
where $\Lambda_x(t)$ and $\Lambda_y(t)$ are real valued scalars. By summing up all the entries of $\frac{d}{dt} \vx(t)$ and $\frac{d}{dt} \vx(t)$ and substituting $\vx(t)^T\ones = \vy(t)^T \ones = 1$, it can be easily seen that $\Lambda_x(t) = \Lambda_y(t) = \kappa \vx(t)^T \vy(t)$. Thus, we retrieve the original system and hence the equivalence. Moving forwards, we will use $\Lambda(t) \triangleq \Lambda_x(t) = \Lambda_y(t) = \kappa \vx(t)^T \vy(t)$ in our equations. Before proceeding with our first important result, we make an observation about fixed points $(\vx^*, \vy^*)$ of \eqref{ODE XY}.
%-----------------------------------------------
\begin{remark}
Any fixed point $(\vx^*,\vy^*) \in S$ of \eqref{ODE XY} has all strictly positive entries. In other words, there exists no $i \in \cN$ such that $x^*_i = 0$ or $y^*_i= 0$.
\label{Remark}
\end{remark}
%-----------------------------------------------
\begin{proof} (Remark \ref{Remark})
Consider $(\vx^*,\vy^*) \in S$ to be fixed points of \eqref{ODE XY}. Then, the $i^{\text{th}}$ entry of $\vx^*$ satisfies the equation
$$0 = \sum_{j \in \cN} Q_{ji}x^*_j + [\kappa {\vx^*}^T \vy^*]x_i - \kappa x^*_i y^*_i.$$
Suppose $x^*_i = 0$. Then the above equation becomes
$0 = \sum_{j \in \cN,~ j \neq i} Q_{ji}x^*_j.$
This means that for any $j \in \cN$ such that $Q_{ji} > 0$ (implying that node $j$ is a neighbor of node $i$), the corresponding entry $x^*_j = 0$. This in turn leads to $x^*_k = 0$, for all neighbors $k$ of node $j$. Since our graph is connected, there is always a path connecting node $i$ to any other node of the graph, implying $\vx = \0$. Similarly, if $y^*_i = 0$, following the same steps as before gives us $\vy^* = \0$. This is in violation of the third equation in \eqref{ODE XY}, giving us a contradiction. This completes the proof.
\end{proof}

%-----------------------------------------------

We are now ready to state the main result connecting the stochastic process from Section \ref{construction} and the deterministic ODE system from Section \ref{ODE intro}.

%--------------------------------------------------------%
\subsection{From stochastic to deterministic dynamics} \label{stoch-deter main}
In this section we show that the stochastic process from Section \ref{construction}, indexed by $n$,  almost surely converges to the deterministic ODE system from Section \ref{ODE intro}. Before we state the theorem, we make the following assumption.
\begin{itemize}
\item [\textbf{A1:}] For any $n \in \N$, $\big( \vx(0), \vy(0) \big) = \big( \vx^n(0), \vy^n(0) \big)$.
\end{itemize}

%-------------------------------------------------------------
\begin{theorem} (Proof in Appendix \ref{stochastic-deterministic}.)
Consider the family of stochastic processes (indexed by $n$) $\big\{ \big( \vx^n(t),\vy^n(t) \big)_{t\geq 0} \big\}_{n \in \N}$, with kernels $\big(\bar \mQ^n\big)_{n\in\N}$ defined as in \eqref{rate x} and \eqref{rate y}. Let $\big(\vx(t),\vy(t)\big)$ be solutions to the ODE system \eqref{ODE F} that satisfy (A1). Then, for all $T \in (0, \infty)$, we have
\begin{equation}
\lim_{n\rightarrow\infty} \sup_{0 \leq t \leq T} \| \big(\vx^n(t), \vy^n(t) \big) - \big(\vx(t), \vy(t) \big) \| = 0 \ \ \ \ \text{a.s.}
\label{convergence a.s.}
\end{equation}
More precisely, for any $\epsilon > 0$ and $ T \in (0,\infty)$, we have
\begin{equation}
\begin{split}
\mP&\Big( \sup_{0 \leq t \leq T} \big\|   \big( \vx^n(t), \vy^n(t) \big) - \big( \vx(t), \vy(t) \big)  \big \| \geq \epsilon\  \Big) \\
&\leq 4N(N-1)\exp\Bigg({-n(1+\kappa)T \cdot h\Big( \frac{\epsilon e^{-MT}}{\sqrt{2} N (N-1) (1+\kappa) T} \Big)}\Bigg),
\end{split}
\label{Prob bound}
\end{equation}
where $h(x) \triangleq (1+x)\log(1+x) - x$, $M$ is the Lipschitz constant from Proposition \ref{Lipschitz prop}, and $N = |\cN|$ is the size of the graph.
\label{fluid limit}
\end{theorem}

%-----------------------------------------------------
\eqref{Prob bound} gives an upper bound on the probability that, uniformly over a finite time horizon $[0,T]$, the stochastic process indexed by $n\in \N$ deviates from the solution to the deterministic process by at least $\epsilon>0$, provided they start at the same initial point.
A discussion on the effect of parameters $n$, $\kappa$ and $N$ on the bound \eqref{Prob bound} and the long run behavior (as $T \to \infty$) is deferred to the end of Section \ref{deterministic-stochastic}.

\section{Analysis of the deterministic flow and convergence to FV} \label{deterministic section}

From Theorem \ref{fluid limit}, the solutions $\big( \vx(t),\vy(t)  \big)_{t \geq 0}$ of the ODE system \eqref{ODE XY} serve as a deterministic approximation for the CTMC $\big( \vx^n(t),\vy^n(t)  \big)_{t \geq 0}$ for larger values of $n\in\N$. Then, it makes sense to analyze the trajectories of the ODE system for its convergence properties and the nature of its fixed points. For the rest of the paper, we use the notation $\Lambda(t) = \Lambda_x(t) = \Lambda_y(t)$ mentioned earlier.

%=======================SUBSECTION=======================
\subsection{Fixed points of the system} \label{fixed-pts}

We first state a consequence of the Courant-Fischer min-max theorem~\cite{Horn,Carl2000} (specific to our case), which we shall refer to later in this Section.
%------------------------------------------
\begin{lemma} Let $\tilde S \triangleq \{ \vw \in \R^N \ \big| \ \vw^T\ones = 0, \| \vw \| = 1\}$. Given a CTMC kernel $\mQ = -\mL$, where $\mL$ is a Laplacian matrix with eigenvalues ordered as $0=\lambda_1 < \lambda_2 \leq \cdots \leq \lambda_N$, we have
\begin{equation}
\lambda_2 = \min_{\vu \in \tilde S} \vu^T[-\mQ]^T\vu,
\label{v2}
\end{equation}
with $\tilde \vecv_2 = \vecv_2/\|\vecv_2 \|$ being the minimizer, and
\begin{equation}
\lambda_k = \max_{\vu \in \tilde S, ~\vu \perp \{\vecv_{k+1}, \cdots, \vecv_N \}} \vu^T[-\mQ]^T\vu,
\label{vk}
\end{equation}
with $\tilde \vecv_k = \vecv_k/\|\vecv_k \|$ being the maximizer.
\label{Rayleigh lem}
\end{lemma}
%--------------------------------------
Let $\vz(t) \triangleq \vx(t) - \vy(t)$ for all $t \geq 0$, and consider an (implicit) system of equations obtained by subtracting the second equation of \eqref{ODE XY} from the first one, which then reads as
\begin{equation}
\begin{split}
\frac{d}{dt} \vz(t) = \mQ^T\vz(t) + \Lambda(t)\vz(t) \\
\vz(t)^T\ones = 0 \ \big( \forall t\geq 0 \big).
\end{split}
\label{ODE Z}
\end{equation}
Let $\Omega$ denote the set of fixed points of \eqref{ODE Z}. Any fixed point $\vz^* \in \Omega$ satisfies
\begin{align}
-\mQ^T\vz^* &= \mL^T \vz^* = \Lambda^* \vz^*
\label{fixed point eqn} \\
{\vz^*}^T \ones &= 0
\end{align}
for some $\Lambda^* \in (0,\infty)$, and could therefore be a (left) eigenvector of $\mL$ (up to a scalar multiple), or the \textit{zero} vector (which we shall often refer to as the \textit{origin}).
Since, for any $t \geq 0$, $\vz(t) = \0$ if and only if $\vx(t) = \vy(t)$, the fixed point $\0 \in \Omega$ of \eqref{ODE Z} corresponds to the invariant set $S_0$ of  \eqref{ODE XY}, defined as
\begin{equation}
S_0 \triangleq \{(\vx,\vy) \in S\ |\ \vx = \vy, \ \vx^T = \vy^T = 1 \}.
\label{S_0}
\end{equation}
To exclude, from our analysis, the case where trajectories of \eqref{ODE Z} might hit \textit{zero} after some finite time, or equivalently trajectories of \eqref{ODE XY} might enter set $S_0$ and stay there for all future times, we show the following result.
%--------------------------------------------------------------
\begin{proposition}
For all sufficiently large $\kappa$, the invariant set $S_0$ defined in \eqref{S_0} is an unstable set for the system \eqref{ODE XY}.
\label{instability x=y}
\end{proposition}
\begin{proof}
Recall the form of the Jacobian of $F:\R^{2N} \rightarrow \R^{2N}$ from Appendix A.2. When evaluated at $\vx = \vy$ (which we shall use here in place of $\vx^*$ and $\vy^*$ to denote fixed points), we can write it as
\begin{equation*}
\J_F(\vx,\vx) =
\begin{bmatrix}
\mQ	& 0\\
0	& \mQ
\end{bmatrix}
+
\kappa \begin{bmatrix}
    \vx \vx^T + \mD_\vx\mD_\vx - \mD_\vx      & \vx \vx^T - \mD_\vx \\
    \vx \vx^T - \mD_\vx       & \vx \vx^T + \mD_\vx\mD_\vx-\mD_\vx
\end{bmatrix}
\end{equation*}

\begin{equation*} =
\begin{bmatrix}
\mQ^T	& 0\\
0	& \mQ^T
\end{bmatrix}
+\kappa
\begin{bmatrix}
\mD_\vx^2	& 0\\
0	& \mD_\vx^2
\end{bmatrix}
+ \kappa
\begin{bmatrix}
1 &	1 \\
1 & 1
\end{bmatrix}
\otimes
\begin{bmatrix}
\vx \vx^T - \mD_{\vx}
\end{bmatrix},
\end{equation*}
Here, $\mD \otimes \mE$ denotes the Kronecker product of two square matrices $\mD$ and $\mE$. To simplify further analysis, we write down the above matrix as $\J_F(\vx,\vx) = \mA + \mB + \mC$, where
\[ \mA =
\begin{bmatrix}
\mQ^T	& 0\\
0	& \mQ^T
\end{bmatrix}, \ \ 
\mB =
 \kappa
\begin{bmatrix}
\mD_\vx^2	& 0\\
0	& \mD_\vx^2
\end{bmatrix}
\]
\[ \mC =
 \kappa
\begin{bmatrix}
1 & 1 \\
1 & 1
\end{bmatrix}
\otimes
\begin{bmatrix}
\vx \vx^T - \mD_{\vx}
\end{bmatrix}
\]
For any symmetric matrix $\M\in\R^{N\times N}$, let $\lambda_1(\M) \leq \lambda_2(\M) \leq \cdots \leq \lambda_N(\M)$ denote the ordering of its eigenvalues. Observe that the eigenvalues of $\mA$ are the same as those of $\mQ$ with double the multiplicity, implying $\lambda_{2N}(\mA)=0$, and $\lambda_{2k}(\mA)=\lambda_k(\mQ)$ for all $k \in \{1,\cdots ,N \}$. The eigenvalues of $\mB$ are the diagonal elements. Thus, $\lambda_1(B) = \kappa \min_{i \in \cN} x_i^2$ and $\lambda_{2N}(B) = \kappa \max_{i \in \cN} x_i^2$. 

For matrix $\mC$, observe that  $\vx\vx^T - \mD_\vx$ has a \textit{zero} row sum, with negative diagonal entries and non-negative off-diagonal entries. It therefore defines a CTMC transition rate matrix, and we have $\lambda_{2N}(\vx\vx^T - \mD_\vx) = 0$, with the other eigenvalues being strictly negative. It can also be checked that the matrix of all ones on the left of the Kronecker product has the spectrum $\{2,0\}$. The eigenvalues of the Kronecker product are equal to eigenvalues of the two involved matrices cross multiplied. We therefore obtain $2N$ eigenvalues of $\mC$, with $\lambda_{2N}(\mC) = \cdots = \lambda_{N-1}(\mC) = 0$, and the others being strictly negative.

The Weyl's inequality\cite{Horn, Carl2000} for real symmetric matrices $\mD,\mE \in \R^{N\times N}$ is given, for any $k \in \{1,\cdots,N\}$, as
\begin{equation*}
\lambda_1(\mD) + \lambda_k(\mE) \leq \lambda_k(\mD+\mE) \leq \lambda_N(\mD) + \lambda_k(\mE)
\end{equation*}
Applying the lower bound of Weyl's inequality on $\lambda_{N-1}(\mA + \mB + \mC)$, we get
\begin{equation*}
\lambda_1(\mA ) + \lambda_{N-1}( \mB + \mC) \leq \lambda_{N-1}(\mA + \mB + \mC).
\end{equation*}
Applying the lower bound of Weyl's inequality once again, this time on $\lambda_{N-1}( \mB + \mC)$, we get
\begin{equation*}
\lambda_1(\mA ) + \lambda_1(\mB ) + \lambda_{N-1}(\mC) \leq \lambda_{N-1}(\mA + \mB + \mC).
\end{equation*}
Now since $\lambda_1(\mA ) = \lambda_1(\mQ) = -\lambda_N(\mL)$, $\lambda_{N-1}(\mC) = 0$ and $\lambda_1(\mB )= \kappa \min_{i \in \cN} x_i^2$, we get
$$-\lambda_N(\mL)+ \kappa \min_{i \in \cN} x_i^2 \leq \lambda_{N-1} (\J_F(\vx,\vx)).$$
As a consequence of Remark ~\ref{Remark} in Section \ref{construction}, $x_i > 0$ for every $i \in \cN$. This means the $N+1$ largest eigenvalues of the Jacobian become positive for all sufficiently large $\kappa>0$, and each fixed point in $S_0$ has an unstable space\footnote{An unstable space of a point is the set where trajectories move away from the point. For individual fixed points, it is usually enough to show that this unstable space is at least $1$-dimensional to guarantee instability. This can be done by showing that one of the eigenvalues of the Jacobian evaluated at the point is positive (for linear instability)\cite{Perko2001}. However when showing instability of all points associated to an invariant, $k$-dimensional set given by $K$, we need to show that the unstable space of these points is of dimension at least $k+1$, in order to rule out the unstable spaces made of vectors pointing simply inside set $K$. Showing the existence of an unstable space of dimension $k+1$ hence shows instability of the set $K$ itself.}
of minimum dimension $N+1$ associated with the linearized system at the point \cite{Perko2001}. Therefore the set $S_0$ is linearly unstable in $S$. This concludes the proof.
\end{proof}
%------------------------------------------------------
The above result is equivalent to saying that the fixed point $\0\in\Omega$  
of \eqref{ODE Z} is unstable. This allows us to consider trajectories starting from $\big( \vx(0), \vy(0) \big) \in S \setminus S_0$, or equivalently, from $\vz(0) \neq 0$, which will be useful in Section \ref{convergence section}.

\subsection{Convergence to the Fiedler vector} \label{convergence section}

From Section \ref{fixed-pts}, we know that the set $\Omega$ of fixed points of \eqref{ODE Z} contains any scalar multiples $c \vecv_k$ of the (left) eigenvectors $\vecv_k$ of $\mL$ for $k \geq 2$. This includes $\0$, which we showed to be an unstable fixed point in Proposition \ref{instability x=y}. The next question is about the convergence properties of the ODE system, and we have the following. 
%------------------------------------
\begin{theorem}
Trajectories of the system \eqref{ODE Z} always converge to a fixed point. Furthermore, the Fiedler vector $\vecv_2$ is an asymptotically stable fixed point of the ODE system (\ref{ODE Z}), with all others being unstable.
\label{convergence to FV}
\end{theorem}
%--------------------------------------------
The rest of this section is devoted to the proof of Theorem~\ref{convergence to FV}. 
%Theorem \ref{convergence to FV} is proved in two steps. 
First, we provide a `Lyapunov function' $V:\R^{N} \rightarrow [0, +\infty)$ with strictly negative slope along trajectories of \eqref{ODE XY}, at any point which is not in $\Omega$ (i.e. not a fixed point). This will subsequently be shown in Proposition \ref{Lyapunov prop}. We then put together our proof of Theorem \ref{convergence to FV}, where we show convergence to a fixed point using the LaSalle invariance principle (Theorem \ref{LaSalle} in Appendix \ref{Lyapunov theory}). This will prove the first statement of
Theorem~\ref{convergence to FV}. To prove the second statement, we use the Lyapunov stability theory to show that all the fixed points of type $c \vecv_k$ ($k\geq 3, c\ne 0$) are unstable, and the only possible candidate left for convergence, i.e., $c \vecv_2$, is asymptotically stable. For the sake of completeness, we provide relevant definitions and results from the theory surrounding LaSalle invariance principle and Lyapunov stability in Appendix \ref{Lyapunov theory}.

Recall that the set $S_0$ is unstable and we only consider trajectories starting from $S \setminus S_0$. For such trajectories, since $\vz(t)\ne \0$ in $S \setminus S_0$, we can 
construct another (implicit) system whose solutions are
\begin{equation}
\tilde \vz(t) \triangleq \frac{\vz(t)}{\| \vz(t) \|} ~~\text{for all}~ t\geq 0.
\label{w system}
\end{equation}
Trajectories of \eqref{w system} are always contained in the set $\tilde S = \{ \vw \in \R^N \ \big | \ \vw^T\ones = 0, \| \vw \| = 1\}$. Non-zero fixed points of \eqref{ODE Z} will also be fixed points of this new system. However, this time, they are normalized, and isolated in space, unique up to only a sign. This will make it easier to apply the LaSalle principle later to show convergence. Since we are now working in the normalized space, we let $\{\tilde \vecv_2, \cdots , \tilde \vecv_N \}$ denote the normalized eigenvectors of $\mL$, and denote the set of fixed points as $\tilde{\Omega} = \{\tilde\vecv_2, \cdots , \tilde\vecv_N \}$.\footnote{Technically, it should be $\tilde{\Omega} = \{\pm\tilde \vecv_2, \cdots , \pm\tilde \vecv_N \}$, but we use $\tilde \vecv_k$ to refer to $\pm\tilde\vecv_k$ irrespective of the sign.}

Define the Lyapunov function $V:\R^N \rightarrow [0,\infty)$ as 
\begin{equation}
V(\vu) = \frac{1}{2}\vu^T[-\mQ]^T\vu = \frac{1}{2}\vu^T\mL\vu, \ \ \ \ \mbox{for any}~ \vu \in \R^N.
\label{V}
\end{equation}
We shall now analyze the function $V(\vu)$ over trajectories $\tilde \vz(t)$ as defined in \eqref{w system}.
\begin{proposition}
For trajectories of (\ref{ODE XY}) starting from any $(\vx(0), \vy(0)) \in S \setminus S_0$, we have $\frac{d}{dt} V(\tilde \vz(t)) \leq 0$, with equality only at the fixed points $\tilde{\Omega}$ that are left eigenvectors of $\mL$.
\label{Lyapunov prop}
\end{proposition}
\begin{proof}
For the rest of this section, we suppress the `$(t)$' notation and assume it implicitly. We also use $\dot f$ to mean $\frac{d}{dt} f$ whenever convenient. Observe that 
\[
V(\tilde \vz(t)) = \frac{1}{2}\tilde\vz^T[-\mQ]^T\tilde \vz = \frac{1}{2}\frac{\vz^T[-\mQ]^T\vz}{\vz^T\vz}. 
\]
Differentiating $V$ along the trajectories of (\ref{ODE Z}) gives
\begin{equation}
\frac{d}{dt} V(\tilde \vz(t)) = \frac{1}{2(\vz^T\vz)^2}\Big( [\vz^T\mQ^T\vz] \big[\frac{d}{dt} \vz^T\vz \big] - [\vz^T\vz] \big[\frac{d}{dt}\vz^T\mQ^T\vz\big]  \Big).
\label{d/dt V}
\end{equation}
First, we have
\begin{equation}
 \frac{d}{dt} \Big(\vz^T\vz\Big) = 2\vz^T \dot \vz(t) = 2\big(\vz^T\mQ^T\vz + \Lambda(t) \vz^T\vz \big),
 \label{d/dt z^Tz}
\end{equation}
and similarly we have
\begin{equation}
\frac{d}{dt} \Big (\vz^T[-\mL]\vz \Big) = \dot \vz^T\mQ^T \vz + \vz^T\mQ^T\dot \vz = 2\big( \Lambda(t)\vz^T\mQ^T\vz + \vz^T[\mQ^2]^T\vz \big).
\label{d/dt z^TLz}
\end{equation}
By substituting (\ref{d/dt z^Tz}) and (\ref{d/dt z^TLz}) into (\ref{d/dt V}), we get
\begin{equation}
\frac{d}{dt} V(\tilde \vz(t)) = \frac{\big(\vz\mQ^T\vz\big)^2 - \big(\vz^T[\mQ^2]^T\vz\big)(\vz^T\vz) }{(\vz^T\vz)^2}
\label{d/dt V - two}
\end{equation}

To show (\ref{d/dt V - two}) is non-positive, we leverage the property of $\mQ = -\mL$ being a symmetric matrix whose eigenvectors $\{ \tilde \vecv_1,\tilde \vecv_2, \cdots \tilde \vecv_N\}$ form an orthonormal basis for $\R^N$. Thus, any vector $\vz \in \R^N$ can be written as a linear combination of these orthonormal eigenvectors. Since the trajectory $\vz$ always satisfies $\vz^T\vecv_1 = \vz^T\ones = 0$, we can write
$\vz =  \sum_{k=2}^N c_k \tilde\vecv_k$, where $c_k = \vz^T \tilde\vecv_k$
and similarly,
\begin{equation}
\vz^T\vz =  \sum_{k=2}^N c_k^2, ~~
\vz^T[-\mQ^T]\vz =  \sum_{k=2}^N \lambda_k c_k^2, ~~
\vz^T[\mQ^2]^T\vz =  \sum_{k=2}^N \lambda_k^2 c_k^2.
\label{decomp}
\end{equation}
Substituting these into (\ref{d/dt V - two}) yields

\begin{equation}
 \frac{d}{dt} V(\tilde \vz(t)) = \Bigg[ \sum_{k=2}^N \lambda_k \frac{c_k(t)^2}{\Big( \sum_{k=2}^N c_k(t)^2\Big)} \Bigg]^2 - \Bigg[  \sum_{k=2}^N \lambda_k^2 \frac{c_k(t)^2}{\Big( \sum_{k=2}^N c_k(t)^2\Big)} \Bigg].
\label{d/dt V - three}
\end{equation}
Now, for each $t > 0$, define a discrete random variable $R(t)$ which takes values $\lambda_k$ with probability $\frac{c_k(t)^2}{\sum_{k=2}^N c_k(t)^2}$ for $k \in \{2, 3, \cdots, N\}$. We can then rewrite (\ref{d/dt V - three}) as
\begin{equation}
 \frac{d}{dt} V(\tilde \vz(t)) = E[R(t)]^2 - E[R(t)^2] = -\text{Var}[R(t)] \leq 0,
 \label{-Var}
\end{equation}
where the equality $\frac{d}{dt} V(\vz(t)) = 0$ holds when the random variable $R(t)$ is constant, i.e., when only one of the $c_k(t)$'s is non-zero. In other words, we have zero derivative only when $\tilde \vz (t)$ hits an eigenvector $\tilde \vecv_k$ for some $k=2, 3, \cdots, N$. This completes the proof.
\end{proof}
%-----------------------------------------------

Using the above results and Theorem \ref{stability prop} from Appendix \ref{Lyapunov theory}, which is regarding the stability of fixed points, we are now ready to prove Theorem \ref{convergence to FV}.
%--------------------------------------------
\begin{proof} (Theorem \ref{convergence to FV})
First, we use the LaSalle invariance principle (Theorem \ref{LaSalle}) to prove convergence to \emph{a} fixed point. Note that any trajectory originating from $\tilde S = \{ \vw \in \R^N \ \big| \ \vw^T\ones = 0, \| \vw \| = 1\}$ (and as a result never leaving $\tilde S$) is relatively compact. This is because $\tilde S$ is a closed and bounded subset of a finite dimensional Euclidean space $\R^N$, and hence compact. Also, the fixed points $\{\tilde \vecv_2, \tilde \vecv_3, \cdots, \tilde \vecv_N \}$ of $\big(\tilde \vz(t) \big)_{t\geq 0}$ are isolated in $\tilde S$. Then, Theorem \ref{LaSalle} implies that the solutions given by \eqref{w system} have all their trajectories starting from $\tilde \vz(0) \in \tilde S$ converge to \emph{a} fixed point (i.e. no limit cycles exist and convergence to a point is guaranteed).

Now, we use Proposition \ref{stability prop}(i) to show that the Fiedler vector $\tilde \vecv_2$ is an asymptotically stable fixed point. We already know from \eqref{v2} in Lemma \ref{Rayleigh lem} that $\tilde \vecv_2$ is the global minimizer of the Lyapunov function $V$, and by Proposition \ref{stability prop}(i), it is asymptotically stable.

To see why the eigenvectors $\tilde \vecv_k \in \tilde \Omega \setminus \{\tilde \vecv_2\}$ are unstable, we show that they can never be local minima. Indeed, for any $\alpha \in (0,1)$ and $k > 2$, let $\vp = \tilde \vecv_k + \alpha(\tilde \vecv_2 - \tilde \vecv_k) = (1-\alpha)\tilde \vecv_k + \alpha\tilde \vecv_2$. Then, we have again as a consequence of \eqref{vk} in Lemma \ref{Rayleigh lem} that $V(\vp) < V(\tilde \vecv_k)$. Moreover, if $\tilde \vz(0) = \vp$ we have $\dot V(\tilde z(0)) < 0$ by Proposition \ref{Lyapunov prop}. Therefore, for $k>2$, $\tilde \vecv_k$ is definitely not a local minimizer, and is thus unstable in view of Proposition~\ref{stability prop}(ii).

In summary, we have shown that the Fiedler vector is the only asymptotically stable fixed point of \eqref{ODE Z}, with all other fixed points (eigenvectors) being unstable. This completes the proof.
\end{proof}
\section{From deterministic back to stochastic dynamics} \label{deterministic-stochastic}
In this section, we take the opportunity to reflect upon the results in Sections \ref{stochastic section} and \ref{deterministic section}, and their influence on the long run behaviour of the stochastic process constructed in Section \ref{construction}. Theorem \ref{fluid limit} made the first connection between stochastic processes $\left( \vx^n(t) , \vy^n(t) \right)$ for $n \in \N$, and the solution $\left( \vx(t) , \vy(t) \right)$ of the deterministic ODE system \eqref{ODE F} (or, consequently, between $\vz^n(t) \triangleq \vx^n(t) - \vy^n(t)$ and $\vz(t)$ in \eqref{ODE Z}) by showing that uniformly over a finite time horizon (and sufficiently large $n$) the stochastic process rarely deviates from the solution of the ODE system. This prompted us to analyze system \eqref{ODE Z} and in Theorem \ref{convergence to FV}, we gave stability results on all the fixed points of \eqref{ODE Z}. The Fiedler vector $\vecv_2$ (up to a constant multiple) turned out to be an asymptotically stable fixed point, with all the higher eigenvectors $\vecv_k$, $k \geq 3$ (up to constant multiples) being unstable. This gives us reason to believe that a working algorithm simulating the interacting stochastic process $\left( \vx^n(t) , \vy^n(t) \right)$ should have $\vz^n(t)$ converge to $c \vecv_2$ in some sense, provided it never gets stuck around other, unstable fixed points forever. In this section, we make this intuition precise by resorting to results in \cite{BW2003}, Sections \ref{ODE intro} and \ref{stoch-deter main}.

For any set $B \subset \R^N$ and given $\vz^n(0) \in B$ (stochastic process starts in $B$), let 
$$T^n(B) = \inf \left\lbrace t \geq 0 \ | \ \vz^n(t) \notin B \right\rbrace$$
define the \emph{exit time} from set $B$. Similarly, let 
$$H^n(B,T) =  \frac{1}{T}\int_0^T \Char_{\{\vz^n(s) \in B \}}ds$$
denote the (random) fraction of time the stochastic process $\vz^n(t)$ spends inside set $B$ in the interval $[0,T]$. Note that $\vz^n(t)$ takes values on the N-dimensional grid ($\Z^N$) scaled by the factor $1/n$. Denote this scaled grid by $\frac{1}{n} \Z^N$. Then, $\vz^n(t) \in B$ for some $B \in R^N$ and $t\geq0$ is possible only if $B \cap \frac{1}{n} \Z^N \neq \varnothing$. Using this, define for every $k \geq 2$ the set $B(\vecv_k,m)$ as the smallest open ball containing $c \vecv_k$, but not the \emph{origin}, such that $B(\vecv_k,m) \cap \frac{1}{m} \Z^N \neq \varnothing$. Note that as $m$ increases, $B(\vecv_k,m)$ gets smaller and smaller for each $k \geq 0$ (since resolution of the grid $\frac{1}{m} \Z^N$ gets finer), until the sets do not intersect anymore, and $\lim_{m \to \infty} B(\vecv_k,m) = \{c \vecv_k, c \in \R \}$\footnote{The above can also be done for the normalized version $\tilde \vz(t)$, in which case the fixed points are $\tilde \vecv_k$ and truly isolated in space, making construction of the isolating simple. However, we construct $B(\vecv_k,m)$  the way we do to emphasize that for larger $n$,  $B(\vecv_k,m)$ can be taken to be tighter around $c \vecv_k$ for larger $m$.}. Therefore for large values of $m$, the sets $B(\vecv_k,m)$ for $k \geq 2$ are disjoint and never contain two or more different eigenvectors as fixed points, thereby isolating them.
Also, for any set $B \subset \R^N$ containing $\vecv_k$, $k \geq 3$, let $\bar B$ represent all points in $B$ such that starting from those points, the ODE trajectories $\vz(t)$ leave $B$ in finite time (i.e. subset of the unstable space of $\vecv_k$ which is contained in $B$). Then, we can prove the following by adapting Propositions 3 and 4 in \cite{BW2003} to our system.
%-------------------------------------------
\begin{proposition} 
Consider $m$ sufficiently large such that $B(\vecv_k,m)$ for all $k \geq 2$ are disjoint. Then, 
\begin{itemize}
\item[(i)]
for any $k \geq 3$ and sufficiently large $n$, the stochastic process originating from $\vz^n(0) \in \bar B(\vecv_k,m)$ leaves the neighborhood $B(\vecv_k,m)$ of $\vecv_k$ in finite time with high probability. More precisely,
$$\mP \left(\limsup_{n \rightarrow \infty} T^n(B(\vecv_k,m)) < +\infty \right) = 1.$$

\item[(ii)]
for $n$ sufficiently large and any $B(\vecv_2,m)$ where $m \leq n$, with high probability the process $\vz^n(t)$ spends almost all its time, in the long run, in the neighborhood $B(\vecv_2,m)$. More precisely, for every $m\in \N$,we have
$$\lim_{n \to \infty} [\liminf_{T \to \infty} H^n(B(\vecv_2,m),T)] = 1 \ \ \ \ a.s.$$
\end{itemize}
\label{Benaim prop}
\end{proposition}
%------------------------------------------------------

The above result states that for all sufficiently large $n$, the stochastic process spends only finite time in a neighborhood of the unstable fixed points $\vecv_k$, $k \geq 3$. It then proceeds to spend almost all its time, in the long run, in a neighborhood of the asymptotically stable fixed point $\vecv_2$. The second statement of Proposition \ref{Benaim prop} is especially helpful to us. Since it implies that for all sufficiently large $n$, the stochastic process spends most of its time in a small neighborhood of $\vecv_2$ with high probability, we expect to observe that $\frac{1}{T}\int_0^T \vz^n(s)ds \approx  \vecv_2$ via simulations, for large time horizon $T$. In the next section, we provide numerical results to show that this is indeed the case for interacting random walks over different values of $n$. 

Lastly, recall the bound \eqref{Prob bound} in Theorem \ref{fluid limit}. \eqref{Prob bound} only gives information on how closely the stochastic process follows the solution to the ODE system, and is not by any means a measure of how quickly $\vz^n(t)$ might converge to $\vecv_2$. That being said, the bound in \eqref{Prob bound} increases with $\kappa$ and $N$, but decreases in $n$. Therefore, theoretically, increases in $\kappa$ and $N$ need to be compensated by increasing $n$ if we want \eqref{Prob bound} to ensure that the stochastic process $\big( \vx^n(t),\vy^n(t) \big)_{t\geq 0}$ closely follows the solution $\big(\vx(t),\vy(t)\big)_{t\geq0}$ of the ODE system \eqref{ODE F}. However in practice our framework turns out to be forgiving and requires only moderately large $n$ to work well with a wide range of $N$ and $\kappa$, as we shall observe in the next section.

\section{Numerical results} \label{Numericals}
We begin by providing our simulation setup in Section \ref{sim setup}. In Section \ref{supporting sim} we present simulation results over a range of parameters and different graphs in order to confirm our main theoretical results from Sections \ref{stochastic section}, \ref{deterministic section} and \ref{deterministic-stochastic}. In Section \ref{dynamic sim} we consider the setting of dynamic graphs, and present numerical results that show the robustness of the performance of framework over a dynamic topology.

\subsection{Simulation setup} \label{sim setup}
We source two real world datasets, one is about a social network of dolphins ($62$ nodes, $159$ edges) \cite{konect:dolphins} and the other is about the Facebook social network ($4039$ nodes, $88234$ edges) from the SNAP repository \cite{snap}.
They are undirected, connected graphs, each having its own combinatorial Laplacian matrix $\mL$. Walking according to $\mQ = -\mL$ on any graph is very natural, since it requires only local information.\footnote{A walker at node $i$, will stay there for a random exponential time with rate $d(i)$, after which it will move to a neighbor of $i$ chosen uniformly at random. Hence, the walker uses only local information to walk according to CTMC $\mQ=-\mL$.} For a specifically chosen $n$, we keep track of the quantities $\hat \vx^n(t) \triangleq \frac{1}{t} \int_0^t \vx^n(s) ds$ and $\hat \vy^n(t) \triangleq \frac{1}{t} \int_0^t \vy^n(s) ds$, the empirical density distributions of type-\emph{x} and type-\emph{y} walkers; and our Fiedler vector estimator $\hat \vz^n(t) \triangleq \hat \vx^n(t) - \hat \vy^n(t)$ for all $t\geq0$.

\begin{figure*}[t!]
  \centering
  \includegraphics[width=\textwidth]{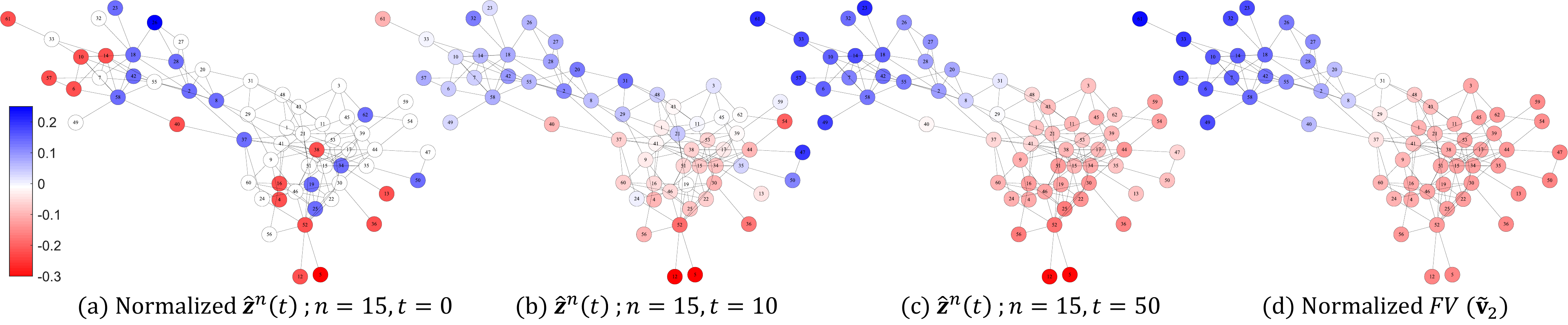}
  \caption{Colormap showing progression of the algorithm over a short time. Red and blue colors denote which group of walkers has majority at the corresponding node, while intensity of the color depends on the value of the corresponding FV estimate. (d) shows partition using the actual Fiedler vector.}
  \vspace{-2mm}
  \label{fig-heatmap}
\end{figure*}

To show how our framework approximates the Fiedler vector and leads to a natural bi-partition of the graph, we illustrate in Figure \ref{fig-heatmap} the progression of a single simulation run over the \emph{Dolphins} graph, with $n=15$ walkers in each group and $\kappa=1000$ kept constant throughout the run. Figures \ref{fig-heatmap}(a), (b) and are `snapshots' taken at times $t=0$ (initial configuration), $10$ and $50$ respectively. The color of a node represents which group of walkers hold majority at that node, and its intensity depends on the corresponding (normalized) value of $\hat \vz^n(t)$. Figure \ref{fig-heatmap}(d) is an illustration of what the spectral bi-partitioning using the true value (computed offline) of $\vecv_2$ would be. As the simulation progresses, we can observe the bi-partition taking shape and Figure \ref{fig-heatmap}(c) and (d) end up closely resembling each other.

We wish to make these observations from Figure \ref{fig-heatmap} more concrete, and show quantitatively that the quantity $\hat \vz^n(t)$ converges to the Fiedler vector $\vecv_2$ over time. For this purpose we use two metrics to help visualize the convergence, namely the Rayleigh quotient ($RQ$) and the cosine similarity ($CS$), which are given respectively by

\begin{align*}
RQ(\hat \vz^n(t)) \triangleq \frac{\hat \vz^n(t)^T\mL\hat \vz^n(t)}{\| \hat \vz^n(t) \|^2}, \ \ \ \ \ \ CS(\hat \vz^n(t)) \triangleq \frac{| \hat \vz^n(t)^T \vecv_2 |}{\| \hat \vz^n(t)  \|\cdot\|\vecv_2 \|}.
\end{align*}
From Lemma \ref{Rayleigh lem}, $RQ(\cdot)$ achieves its minimum, $\lambda_2$ at the Fiedler vector $\vecv_2$. Therefore $RQ(\hat \vz^n(t))$ approaching $\lambda_2$ signifies that $\hat \vz^n(t)$ aproaches $\vecv_2$. The cosine similarity $CS(\hat \vz^n(t))$ tracks the angle between  $\vecv_2$ and $\hat \vz^n(t)$. $CS(\hat \vz^n(t)) = 1$ if $\hat \vz^n(t)$ aligns with $\vecv_2$, implying its convergence, and $CS(\hat \vz^n(t)) = 0$ if it is orthogonal to $\vecv_2$. $CS(\hat \vz^n(t)$ does not just track the convergence, but also tests Proposition \ref{Benaim prop}(i), i.e. whether $\hat \vz^n(t)$ ever gets stuck near one of the unstable fixed points, $\vecv_k$ for $k \geq 3$. Such a case would cause $CS(\hat \vz^n(t)) = 0$, and will be fairly visible in our results if it is a common occurrence. For each simulation run, every random walker starts at an initial position on the graph selected uniformly at random. We use 100 different runs for each simulation setup and display results as averaged over the 100 runs. To test these simulations, we pre-compute values of $\lambda_2$ and $\vecv_2$ for our datasets and use them to compare with $RQ$ and to obtain values of $CS$ respectively. It should be noted that this is purely for comparison purposes, and our framework never requires such computations to be carried out.

\subsection{Simulation results over a range of parameters} \label{supporting sim}

\begin{figure*}[t!]
    \centering
    \vspace{-2mm}
    \hspace{-1mm}
    \subfigure[$RQ$ vs $t$ over different $n$; $\kappa = 1000$ (inset plotted on log-log scale)]{\includegraphics[scale = 0.35]{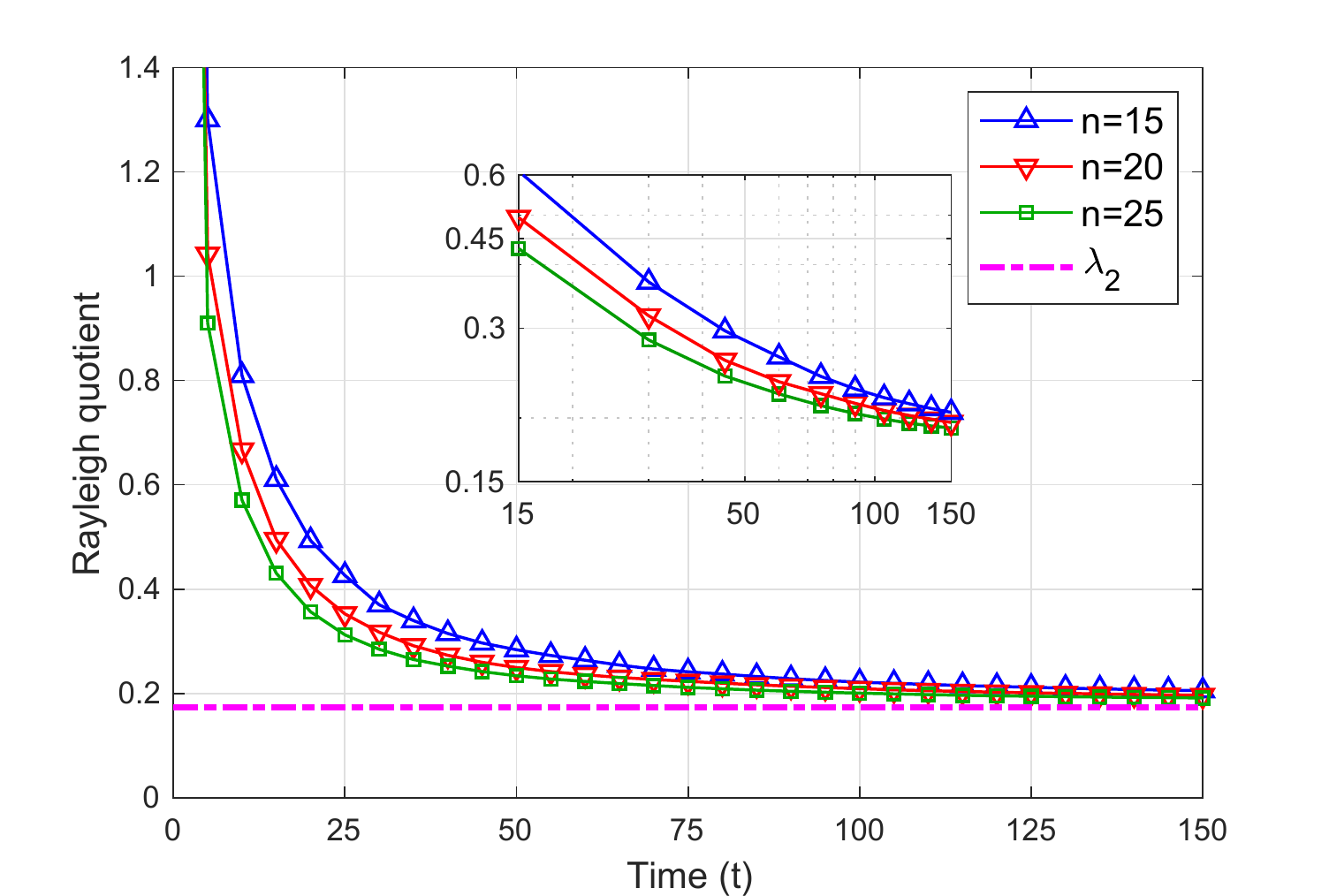}}
    \hspace{-1mm}
    \subfigure[$CS$ vs $t$ over different$n$; $\kappa = 1000$]{\includegraphics[scale = 0.35]{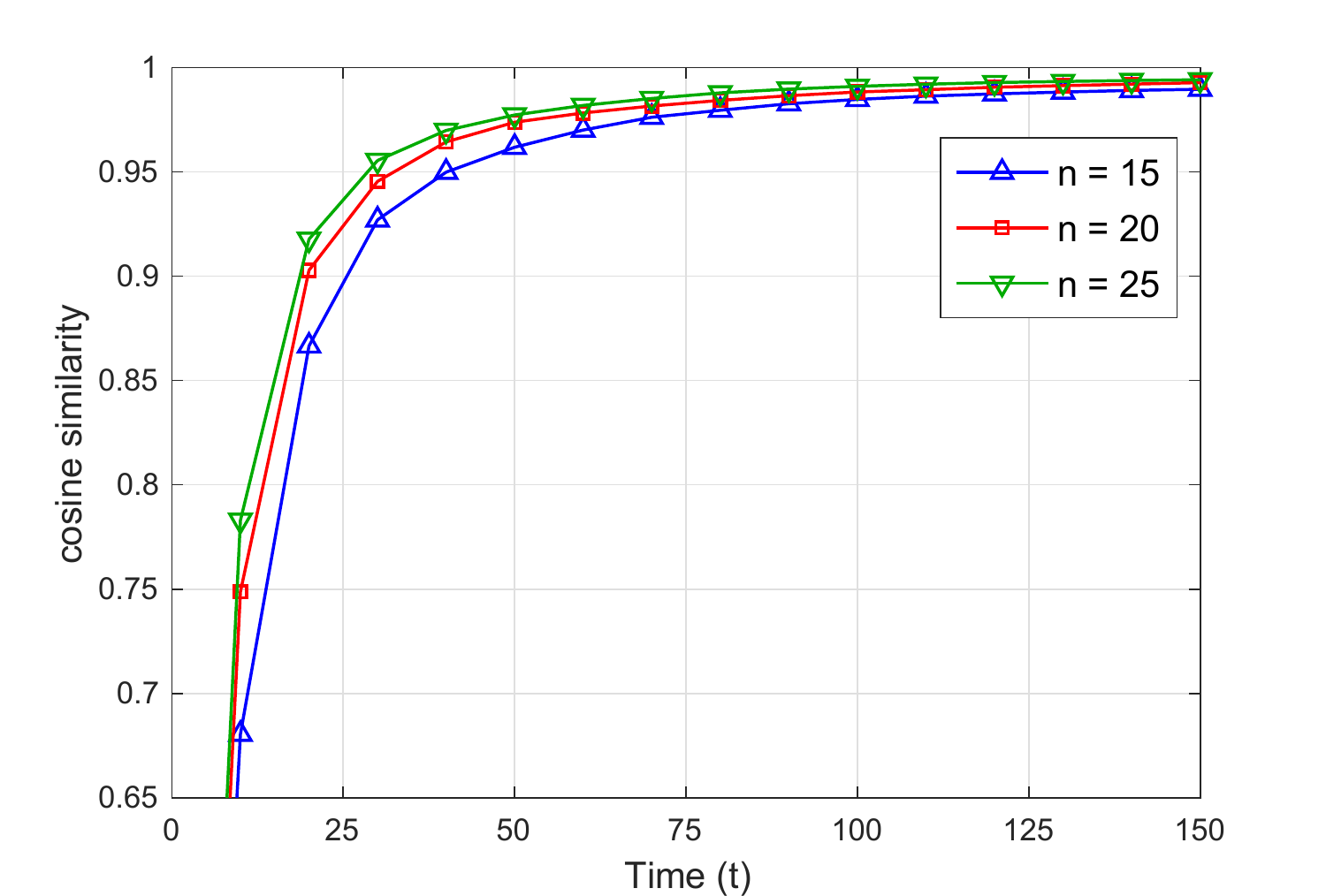}}
    \hspace{-1mm}\\

    \hspace{-1mm}
     \subfigure[$RQ$ vs $t$ over different $\kappa$; $n = 15$]{\includegraphics[scale = 0.35]{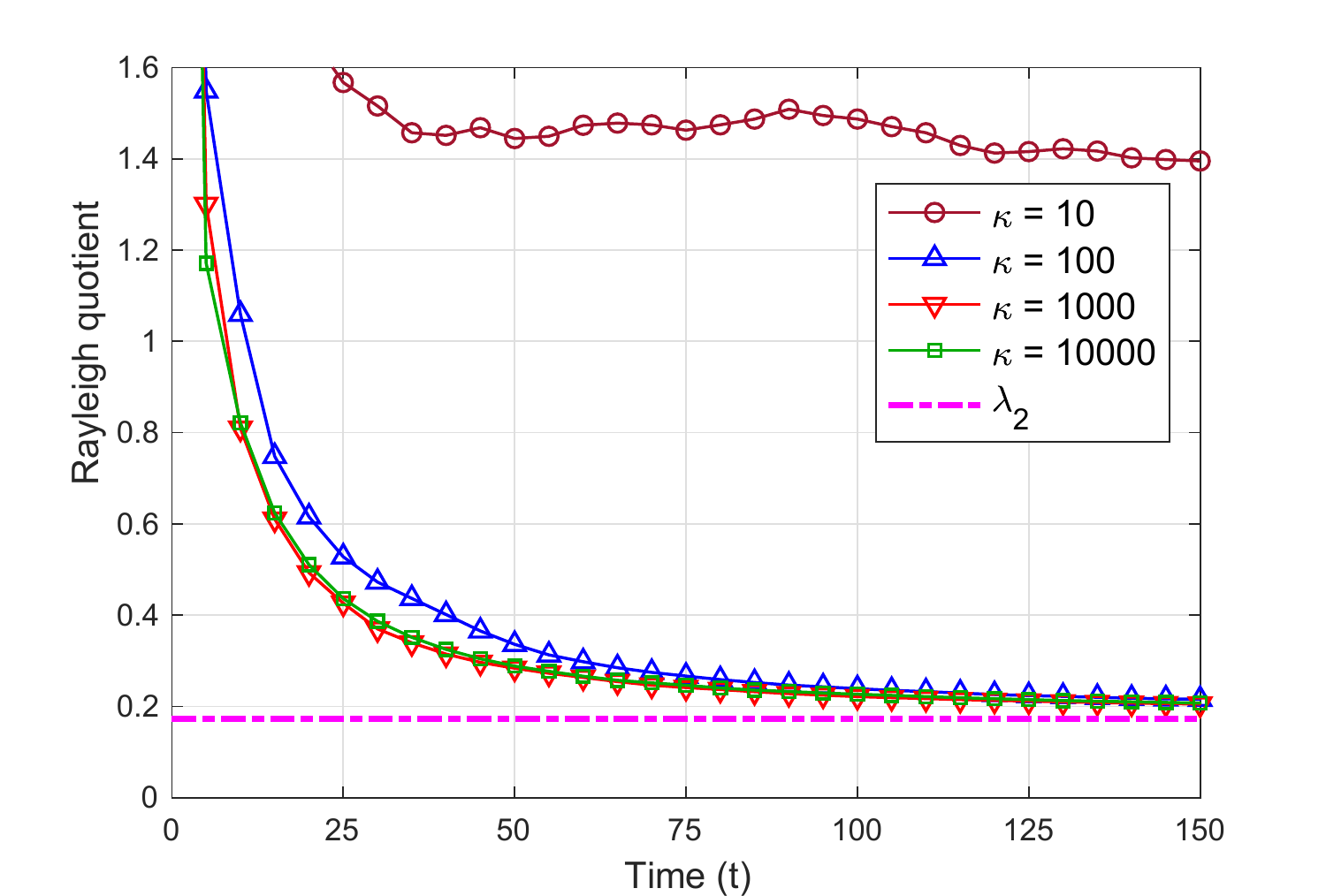}}
    \vspace{-0mm}
    \hspace{-1mm}
     \subfigure[$CS$ vs $t$ over different $\kappa$; $n = 15$]{\includegraphics[scale = 0.35]{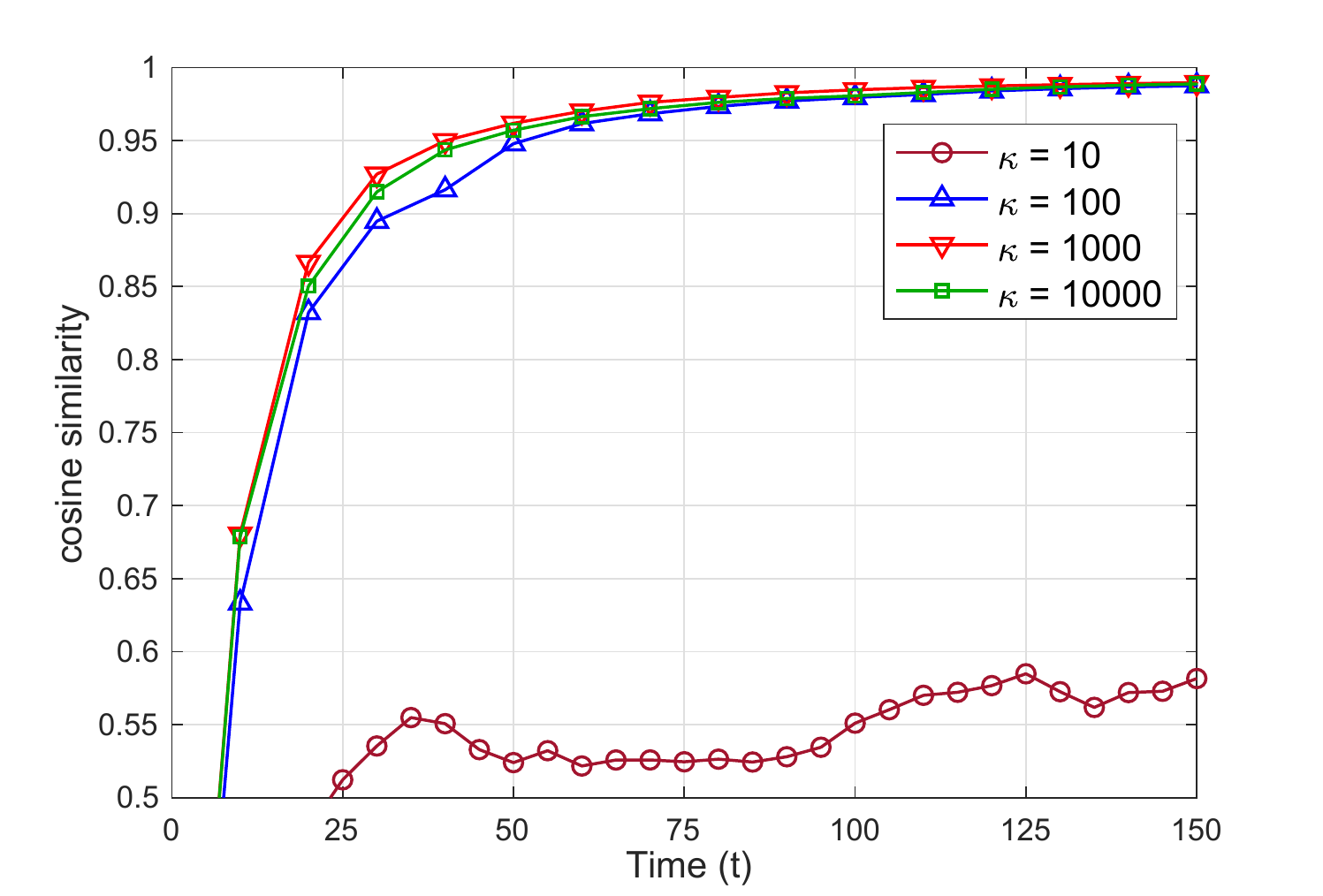}}
     \hspace{-1mm}
    \vspace{-3mm}
    \caption{Simulations for the \emph{Dolphins} graph.}
    \vspace{-3mm}
    \label{dolp}
\end{figure*}

We first present simulation results for the \emph{Dolphins} graph. Figures \ref{dolp}(a) and (b) show the results for values of $n = 15, 20$ and $25$ for each type of walker, with $\kappa$ set at $1000$. For each $n$, simulation shows the Rayleigh quotient $RQ$ approaching $\lambda_2$ of the Laplacian of the graph. This hints convergence to $\vecv_2$, which is also confirmed by the cosine similarity $CS$ approaching \emph{one}. The inset in Figure \ref{dolp}(a) plots the data on the log-log scale, and shows that $RQ$ decays roughly polynomially fast to $\lambda_2$ for each $n$. As would be expected, $\vz^n(t)$ for $n=25$ outperforms the other two choices of $n$ early on in the simulation, but its early advantage in performance is less pronounced as the simulation progresses.

Figures \ref{dolp}(c) and (d) show convergence results under different values of $\kappa = 10, 100, 1000$ and $10000$, while keeping $n=15$ for each run. Observe that the simulation results for $\kappa = 10$ do not show any convergence to the Fiedler vector. This can be attributed to the failure of $\kappa = 10$ in satisfying Proposition \ref{instability x=y} by not being sufficiently large. For larger values of $\kappa$, we observe robust convergence over a wide range of values with no visible trend in $\kappa$.
\begin{figure*}[t!]
\centering
    \vspace{-2mm}
    \hspace{-1mm}
    \subfigure[$RQ$ vs $t$ over different $n$; $\kappa = 10000$ (inset plotted on log-log scale)]{\includegraphics[scale = 0.35]{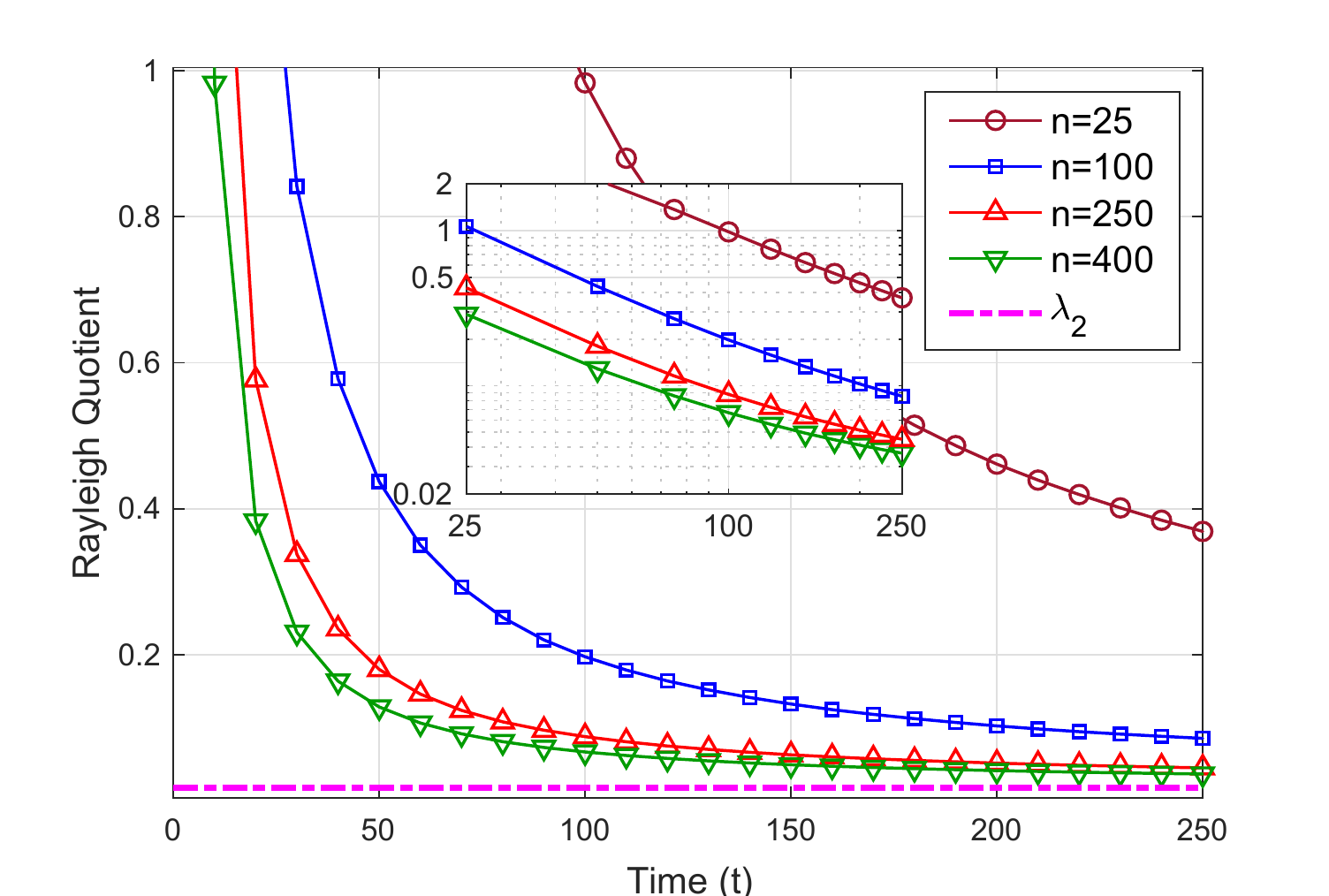}}
    \hspace{-1mm}
    \subfigure[$CS$ vs $t$ over different $n$; $\kappa = 10000$]{\includegraphics[scale = 0.35]{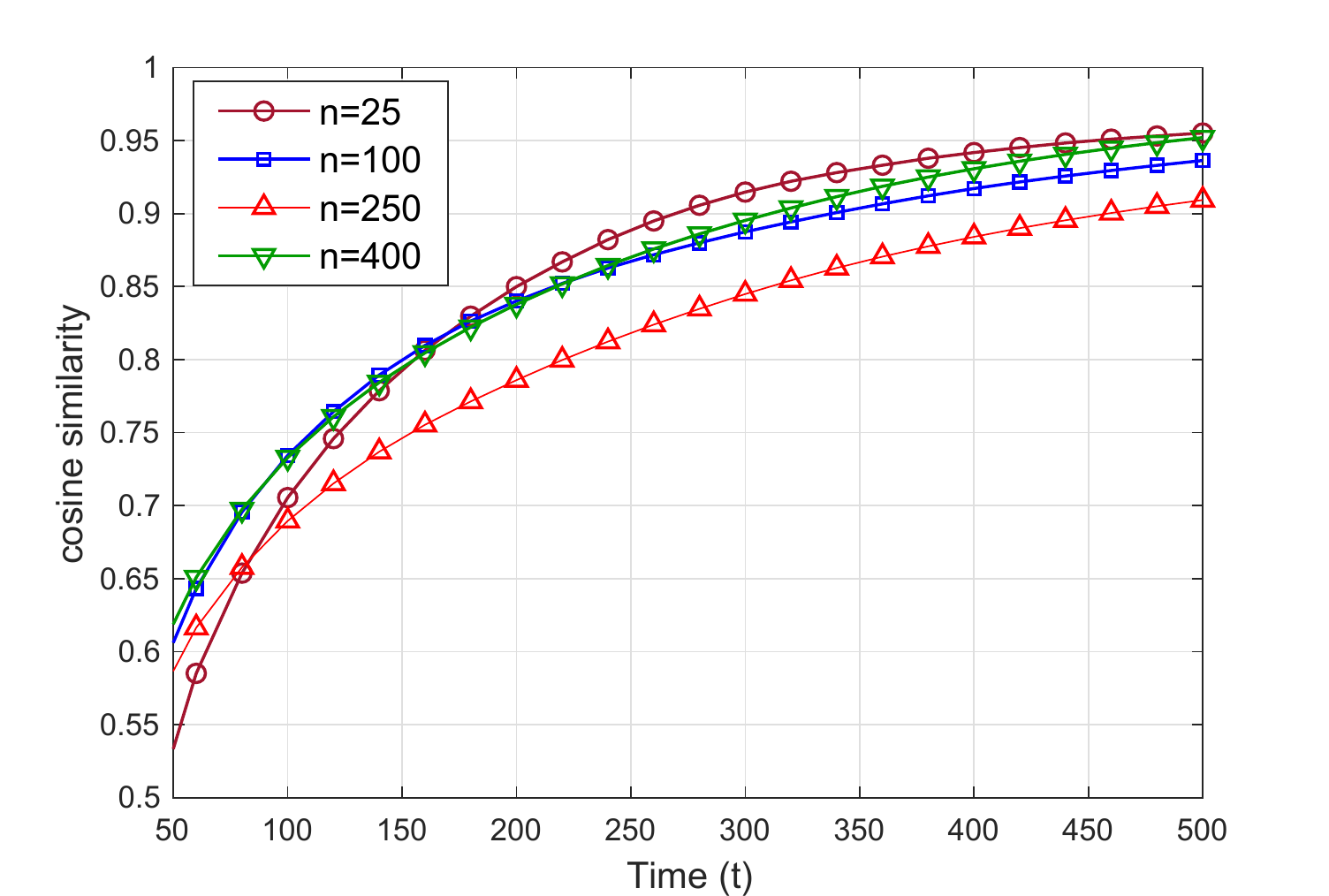}}
    \hspace{-1mm}
    \vspace{-0mm}\\

    \subfigure[$RQ$ vs $t$ over different $\kappa$; $n = 100$]{\includegraphics[scale = 0.35]{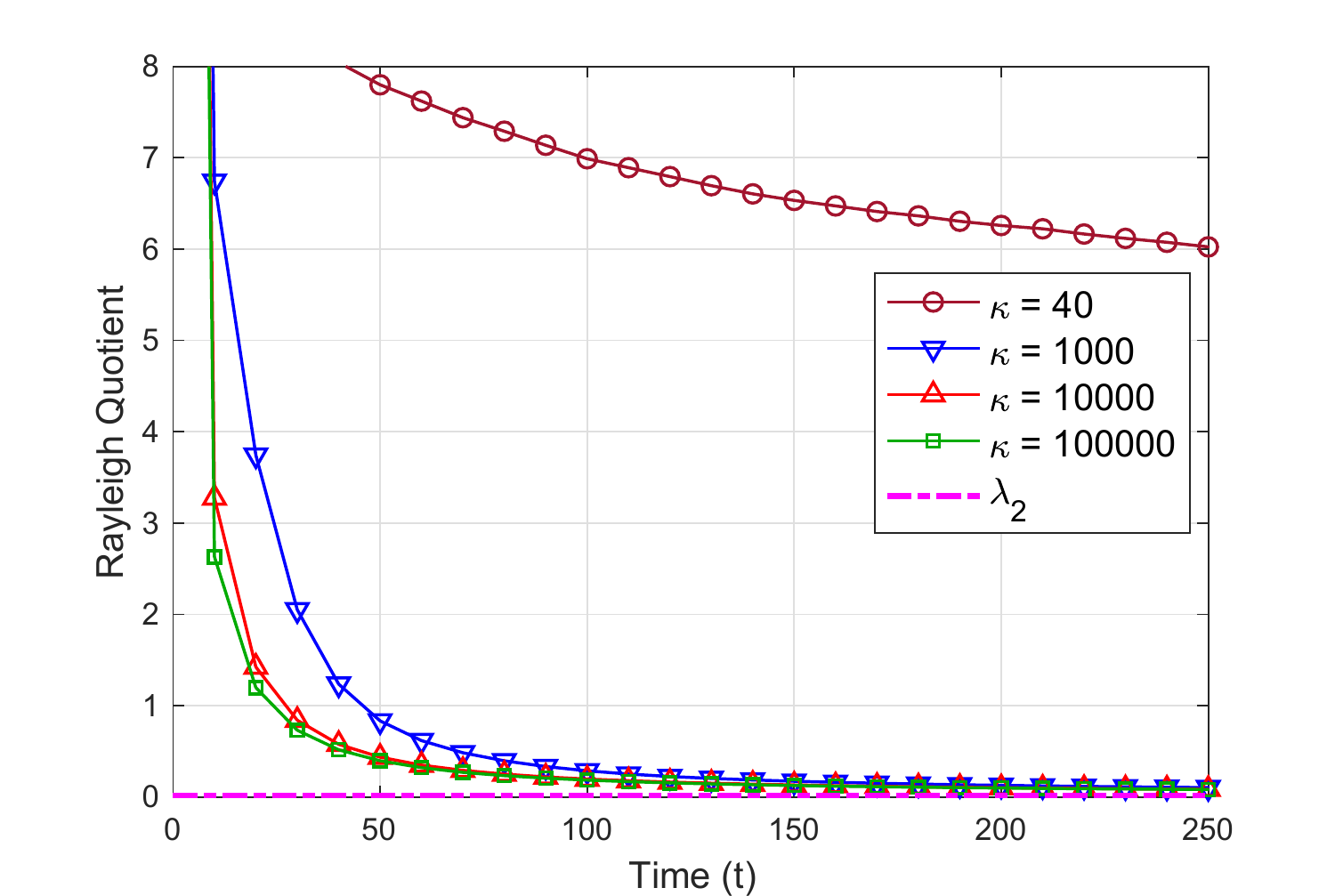}}
    \vspace{-0mm}
    \hspace{-1mm}
    \subfigure[$CS$ vs $t$ over different $\kappa$; $n = 100$]{\includegraphics[scale = 0.35]{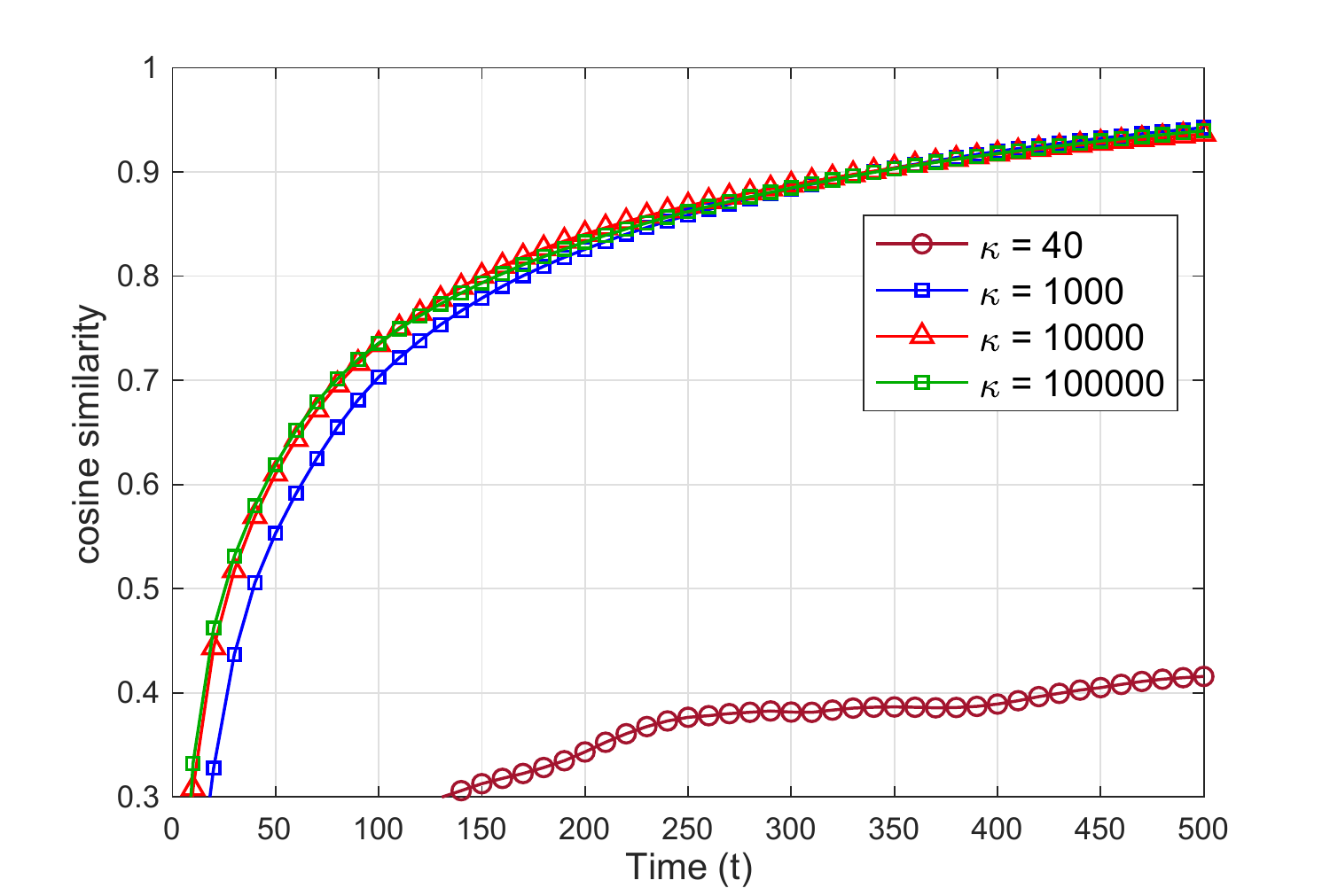}}
    \hspace{-1mm}
    \vspace{-3mm}
    \caption{Simulations for the \emph{ego-Facebook} graph.}
    \vspace{-3mm}
    \label{face}
\end{figure*}
We now present simulation results for the \emph{ego-Facebook} graph, where we perform similar simulations as before. it is interesting to see how well our interacting random walk based method scales to a larger graph. Figures \ref{face}(a) and (b) show simulation results for $n=25, 100, 250$ and $400$, for a fixed $\kappa = 10000$. As can be seen in Figure \ref{face}(a), the Rayleigh quotient for all four simulation sets converges towards $\lambda_2$. The inset in Figure \ref{face}(a) shows that similar to results for the \emph{Dolphins} graph, $RQ$ decays roughly polynomially to $\lambda_2$. The simulation for $n=25$ seems to perform worst, at least early on, and $n=400$ performs best early on. An increase in $n$ clearly shows faster convergence of $RQ$, but the big improvement from $n=25$ to $n=100$ is nowhere to be seen between $n=250$ and $n=400$. This appears to be because the probability of the stochastic process of deviating from the deterministic ODE decays exponentially fast in $n$, as theorized in Theorem \ref{fluid limit}. This is also the reason why even though \emph{ego-Facebook} is a much larger graph than \emph{Dolphins}, even values of $n$ that are less that $3\%$ of the graph size can provide good approximations, compared to about $50\%$ needed for the much smaller \emph{Dolphins} graph. In the longer run, $\hat \vz^n(t)$ for all the simulations seems to align well enough with $\vecv_2$, as can be seen in Figure \ref{face}(b). Even for small number of walkers ($n=25$), the $CS$ eventually catches up with the rest and shows the expected close alignment with $\vecv_2$. This suggests that if high precision approximation is not needed, lower values of $n$ can provide quick, yet reasonably accurate approximations of the Fiedler vector. This is useful for applications such as graph partitioning, where only the signs of the entries of the Fiedler vector matter, not the numerical entries themselves.

We also present simulations over \emph{ego-Facebook} to test the effect of $\kappa$. Figures \ref{face} (c) and (d) show the results for $\kappa = 40, 1000,10000$ and $100000$, keeping $n=100$ constant over all the run\footnote{We plot these on a different scale compared to Figures \ref{face}(a) and (b) to accommodate all the results in a single frame.}. We can observe that $\kappa = 40$ fails to be sufficiently large to satisfy Proposition \ref{instability x=y} and does not show convergence to $\vecv_2$. As observed before for the \emph{Dolphins} graph, there is no visible trend when it comes to $\kappa$ influencing the performance of the simulation, with similar performance for a wide range of higher values of $\kappa$. Overall from simulation results of both the \emph{Dolphins} and \emph{ego-Facebook} graphs, it appears that as long as $\kappa$ is large enough to satisfy Proposition \ref{instability x=y} (which it seems to satisfy as long as $\kappa$ is atleast of order $o(N)$) its value does not affect the performance of our framework, with convergence being robust for any choice of large $\kappa$.

It should be noted that in no simulated case did large $N$ or large $\kappa$ affect the ability of our stochastic process $\big( \vx^n(t),\vy^n(t) \big)_{t\geq 0}$ to closely follow the solutions $\big( \vx(t),\vy(t) \big)_{t\geq 0}$ of our ODE system \eqref{ODE F}, even though they possibly could have caused the bound in \eqref{Prob bound} from Theorem \ref{fluid limit} to grow large\footnotemark. This is because \eqref{Prob bound} is just an upper bound and in reality, the $\big( \vx^n(t),\vy^n(t) \big)$ could closely follow $\big( \vx(t),\vy(t) \big)$ over a much larger range of parameters. Moreover, it is an upper bound for the largest deviation over a finite time horizon $[0,T]$ and not for each time $t>0$, meaning that it provides a worst case bound and efforts taken to make it smaller can be overkill.

\footnotetext{Failure to converge for smaller values of $\kappa$ in Figures \ref{dolp}(c) and (d), and Figures \ref{face}(c) and (d) is attributed to failure to satisfy Proposition \ref{instability x=y}, and is an intrinsic trait of the ODE system \eqref{ODE F} itself and not the stochastic process.}

\subsection{Simulations on dynamic graphs} \label{dynamic sim}
As briefly mentioned in Section \ref{intro}, our random walk based framework is expected to perform well in the setting of dynamic graphs. In this section we provide simulation results to support that statement. We modify the \emph{Dolphins} and \emph{ego-Facebook} graphs by removing sets of nodes at certain time points during the course of each simulation run. When a node is removed, we also delete all the edges associated with it. While the instances at which these changes occur are selected for convenience of presentation, the sets of nodes to be removed are randomly generated before the simulation begins, with the condition that their removal does not affect the connectivity of the original graph. To differentiate the dynamic graphs from the original one, we rename our dynamic version of the graphs with node removal as \emph{Dolphins-dyn} and \emph{ego-Facebook-dyn}, respectively. We change the graphs three times in total by deleting a different set of nodes and all related edges at each of those instances. Overall, we delete around $21.0\%$ of nodes from the original \emph{Dolphins} graph and $12.4\%$ of nodes from the original \emph{ego-Facebook} graph, maintaining their connectivity as we do so. Table \ref{dynamic graph table} gives details about the times at which nodes are removed and the number of nodes removed.

\begin{table}[htbp]
\centering
%\small
\begin{tabular}{|c|c|c|c|c|c|}
\hline
 & ($t_1$, $\#$ nodes removed) & ($t_2$, $\#$ nodes removed) & ($t_3$, $\#$ nodes removed) \\ \hline
Dolphins-dyn  &(25, 4 nodes)    &(75, 3 nodes)		&(100, 6 nodes)	 \\ \hline
ego-Facebook-dyn  &(50, 196 nodes)    &(100, 150 nodes)		&(125, 154 nodes)    \\ \hline
\end{tabular}
\vspace{1mm}
\caption{Statistics of the dynamic changes made}\label{dynamic graph table}
\vspace{-5mm}
\end{table}

During the three pre-selected times $t_1, t_2, t_3 > 0$ when all the nodes from the corresponding sets are removed, any random walker currently located at the removed node is redistributed to the position of another randomly selected walker (of the same group) present at an unremoved node of the graph\footnote{We did this only for convenience of the simulation, and the redistribution need not follow a prescribed rule since our theory allows for convergence starting from any arbitrary initial condition.}. While plotting the Rayleigh quotient for the dynamic graphs, we represent by the black dotted lines the values of $\lambda_2(\cG)$, the algebraic connectivity of $\cG$, updated to reflect the changes in graph topology.

\begin{figure*}[t!]
\centering
    \vspace{-2mm}
    \hspace{-0mm}
    \subfigure[$RQ$ vs $t$ for \emph{Dolphins}-dyn]{\includegraphics[scale = 0.35]{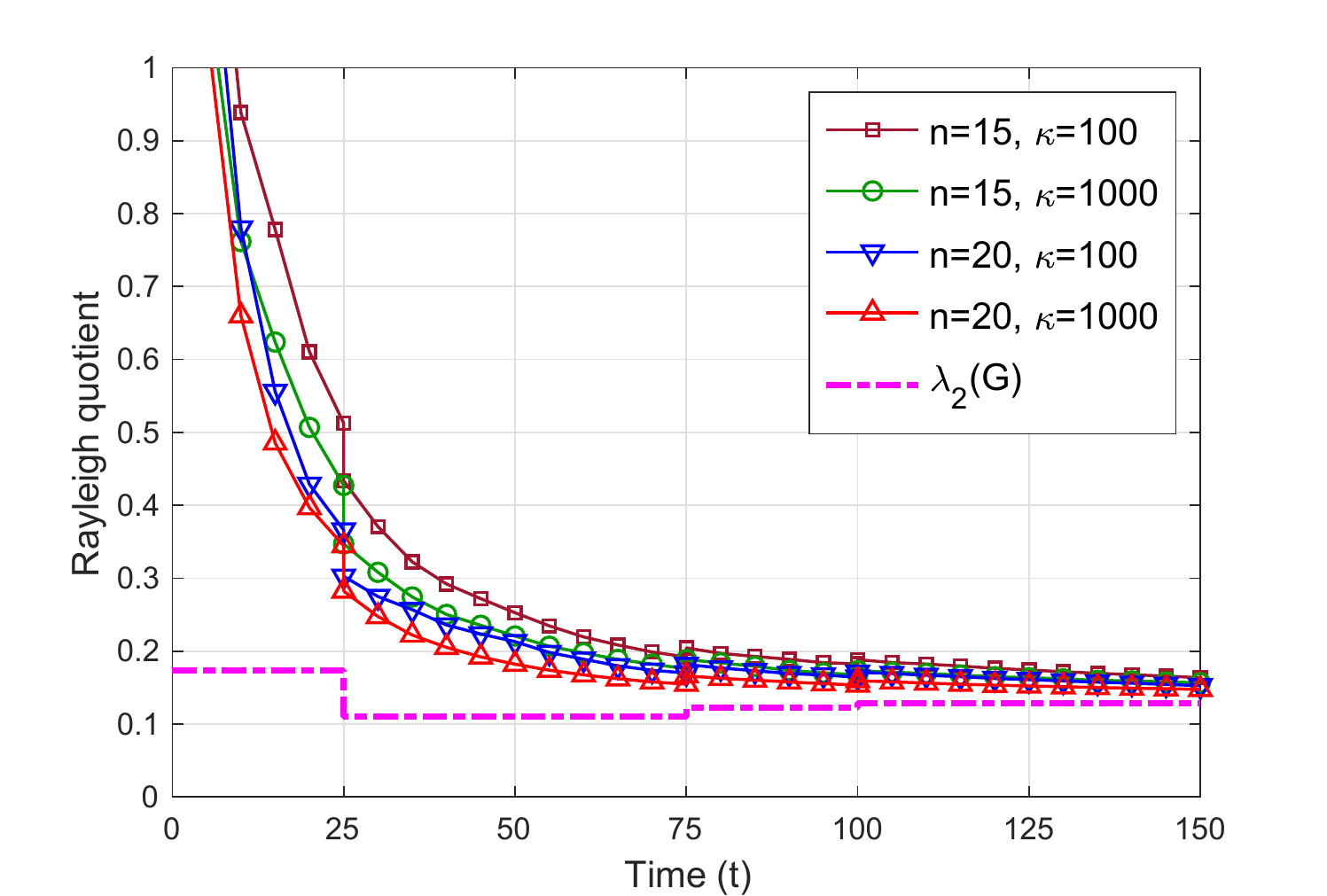}}
    \hspace{-0mm}
    \subfigure[$RQ$ vs $t$ for \emph{ego-Facebook}-dyn]{\includegraphics[scale = 0.35]{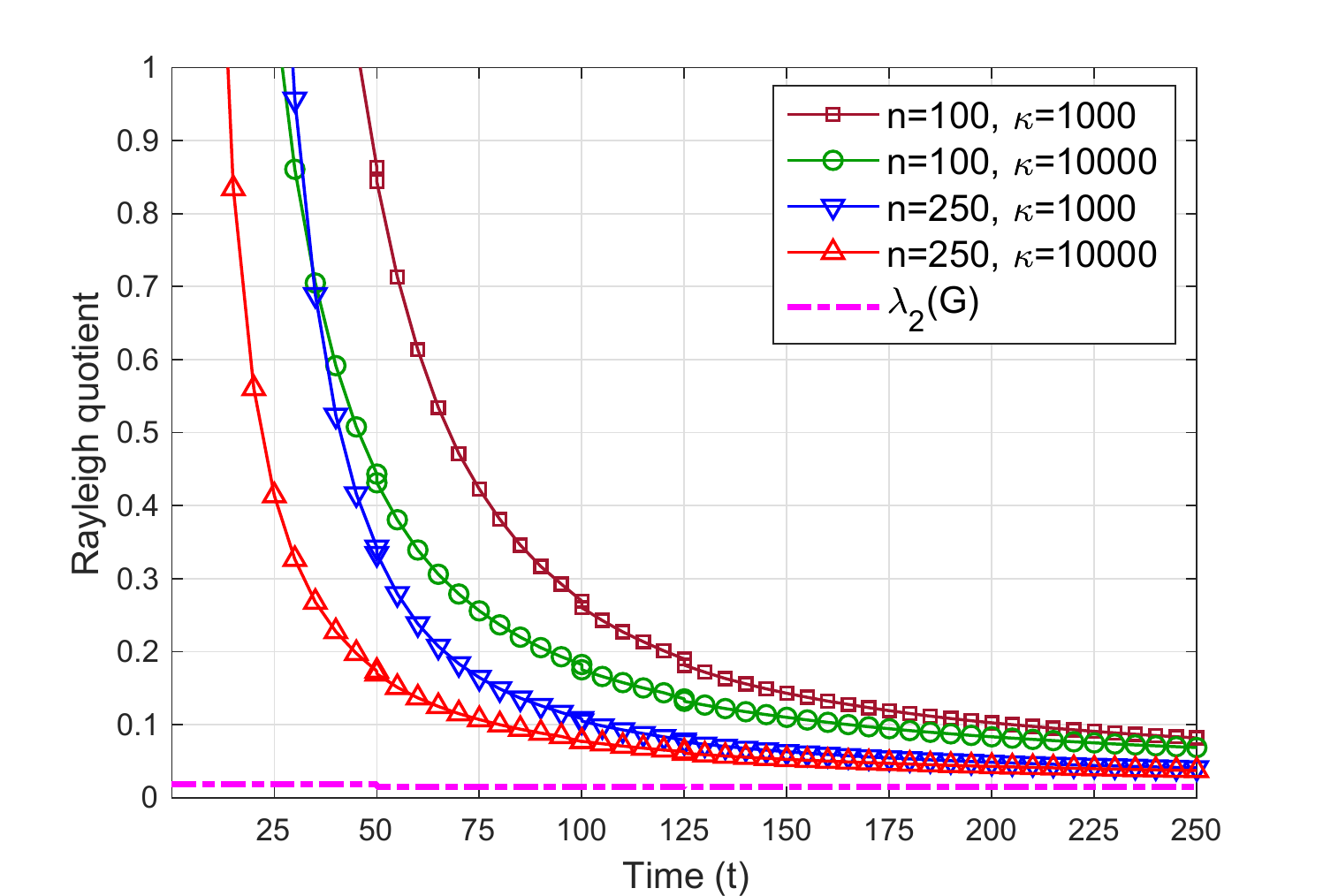}}
    \hspace{-0mm} \\

    \subfigure[$CS$ vs $t$ for \emph{Dolphins}-dyn]{\includegraphics[scale = 0.35]{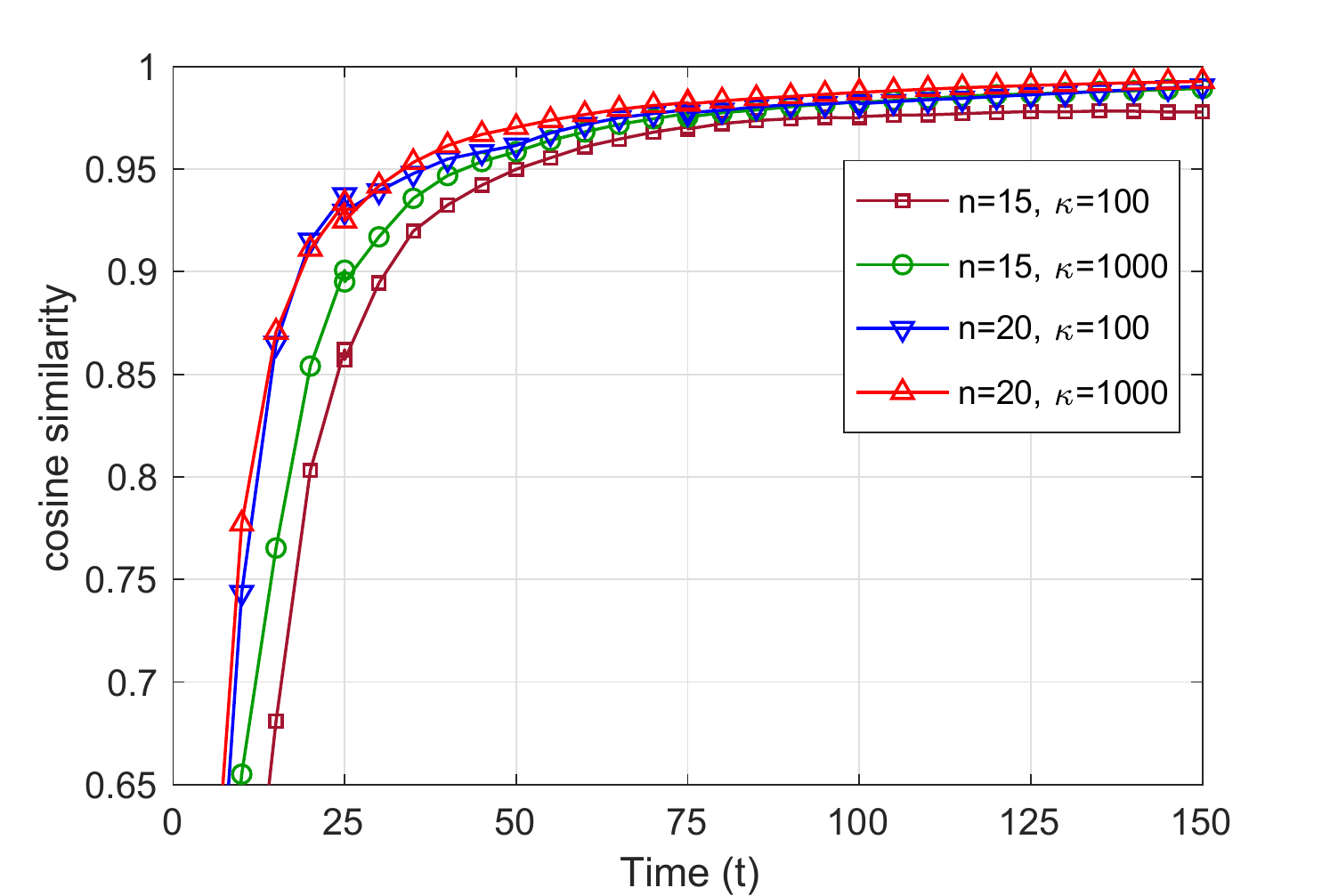}}
    \vspace{-0mm}
    \hspace{-0mm}
     \subfigure[$CS$ vs $t$ for \emph{ego-Facebook}-dyn]{\includegraphics[scale = 0.35]{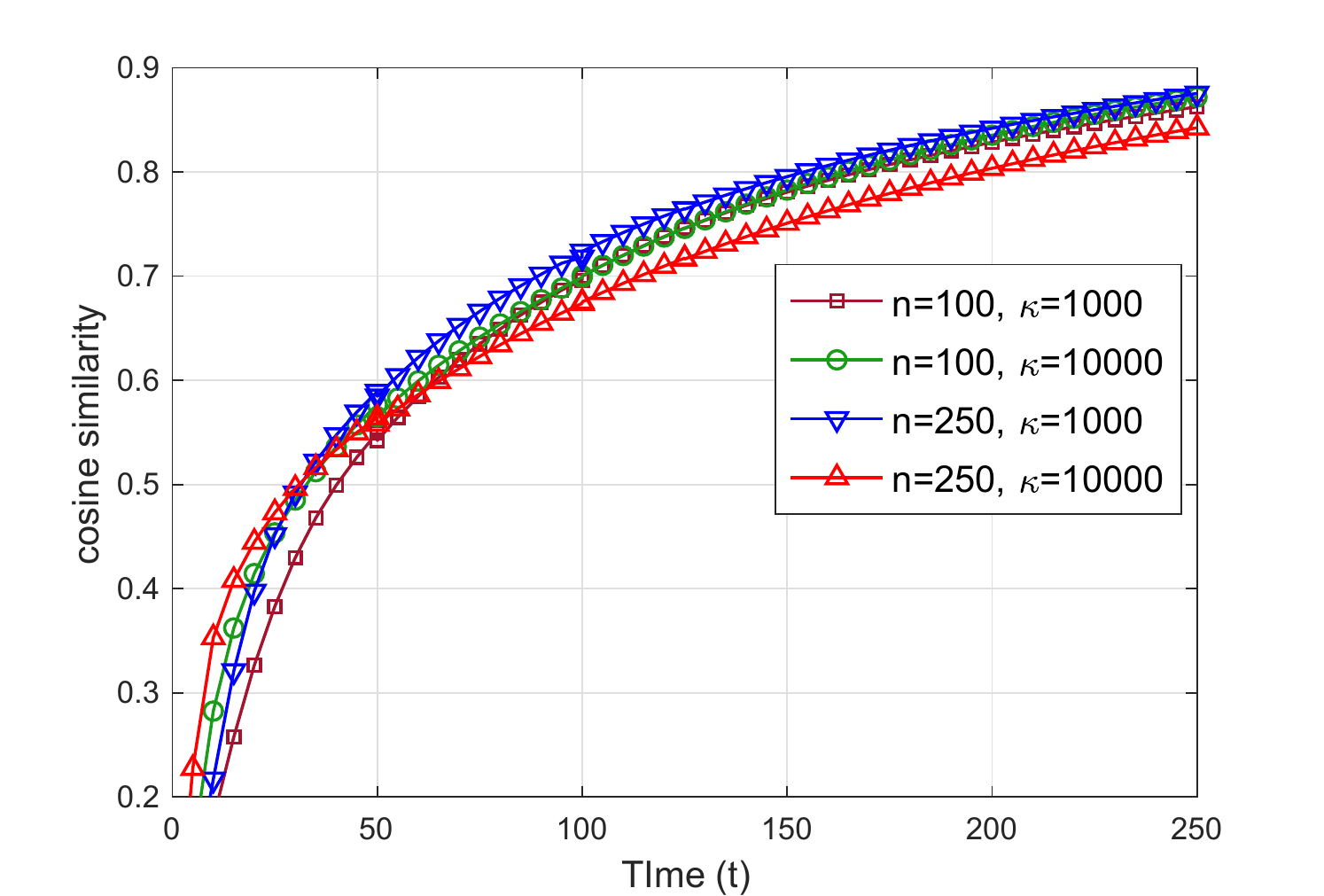}}
    \vspace{-3mm}
    \caption{Simulations for \emph{Dolphins} and \emph{ego-Facebook} graphs subject to node removals.}
    \vspace{-5mm}
    \label{dynamic_figs}
\end{figure*}

Figures \ref{dynamic_figs}(a) to (d) show our simulation results under the dynamic graph setting over a range of parameters. Figures \ref{dynamic_figs}(a) and (b) show results concerning Rayleigh quotient $RQ$ for the \emph{Dolphins-dyn} and \emph{ego-Facebook-dyn} graphs. We represent by the black dotted lines, the values of $\lambda_2$ updated to reflect the modification in graph topology. As expected, we observe small jumps in $RQ$ at the times when nodes are deleted. Better visible in Figure \ref{dynamic_figs}(a) than in Figure \ref{dynamic_figs}(b), these jumps are similar in size to the jumps in $\lambda_2(\cG)$, and apart from these, the Rayleigh quotient values monotonically decrease with time without any unexpected fluctuations. From this, we can safely conclude that even after nodes are deleted, trajectories of the simulation do not diverge from their intended paths and keep converging to the Fiedler vector (of the updated/modified graphs). While they do not diverge, the trajectories do face minor changes in terms of a small loss in progress, as can be observed across all the values of $n$ and $\kappa$ in Figure \ref{dynamic_figs}(c) at time $t=25$, where the cosine similarity $CS$ for \emph{Dolphins-dyn} drops due to deletion of nodes. On the other hand, these changes can also be minor improvements/gains in progress as can be seen more clearly for $n=250, \kappa=1000$ in Figure \ref{dynamic_figs}(d) at time $t=50$, where $CS$ for \emph{ego-Facebook-dyn} actually increases due to deletion of nodes. Therefore changes in topology can randomly be either beneficial, or disadvantageous depending on the random state the simulation is in. However, these random effects seem too minor to be a cause of concern, and there is no observable improvement or reduction in performance.

Thus, we can conclude that the framework behaves as would be expected and is robust to dynamical changes in topology, as long as the graph remains connected. It should be noted that even if the graph gets disconnected from time to time, the framework will still, by its design, converge to the Fiedler vector as long as connectivity is eventually restored again\footnote{For a disconnected graph, the Fiedler vector is not well defined, since the second eigenvalue is also \emph{zero}, and no longer strictly positive. Therefore talking about the Fiedler vector of the whole graph makes sense only when it is connected.}.

\section{Extensions to reversible Markov Chains} \label{time reversible}
In Section \ref{stochastic section} the CTMC used as the basis for constructing our stochastic process was given by $\mQ = -\mL$, the negative of the combinatorial Laplacian. We then went on the show convergence of a suitably scaled version of this process to the Fiedler vector $\vecv_2$ of $\mL$ in Section \ref{deterministic section} and discussed the long run behavior of our stochastic process in Section \ref{deterministic-stochastic}. In this section, we generalize our results by extending them to not just the combinatorial Laplacian $\mL$, but also any time reversible CTMC kernel which, as mentioned in Section \ref{intro}, is an important part of our paper's contribution.

Consider the kernel $\mQ \in \R^{N \times N}$ of an ergodic time reversible CTMC on a finite state space $\cN$ (where $|\cN| = N$), not necessarily symmetric. We denote by $\uppi \in \R^N$, its stationary distribution ($\uppi^T\mQ = 0$).  The vector $\uppi$ is also the first left eigenvector of $\mQ$. If we now define our stochastic process in Section \ref{construction} and all the consequent systems with respect to time reversible $\mQ$, we still have $\vz^n(t)^T\ones = 0$ (and $\vz(t)^T\ones = 0$) by construction. Earlier in Section \ref{stochastic section} when $\mQ$ was a symmetric matrix, its eigenvectors formed an orthogonal basis of the $N$-dimensional Euclidean space. This meant that being orthogonal to $\ones$ automatically guaranteed being in a set where $\vecv_2$ minimized the Lyapunov function $V$. This, however, is not the case for any general non-symmetric, time reversible CTMC kernel $\mQ$. We will extend our results to time reversible Markov chain kernels by introducing a specially constructed inner product which ensures orthogonality to the first left eigenvector of $\mQ$, $\uppi.$

For these purposes, we first formally define time reversible Markov chains.

%-------------------------------------
\begin{definition}
A Markov chain with $\mQ$ is called `time reversible' if and only if there exists a unique $\uppi$ such that for all $i,j \in \cN, i \neq j$, the pair $(\uppi,\mQ)$ satisfies the `detailed balance equation'
$\pi_i Q_{ij} = \pi_j Q_{ji}.$
\label{DBE}
\end{definition}
%-------------------------------------
For a given pair $(\uppi, \mQ)$, define an $N$-dimensional vector space $\cH_{\frac{1}{\pi}}$ endowed with an inner product $\innprod{\cdot}{\cdot}$, defined as
$$\innprod{\vx}{\vy} = \vx^T \Pi^{-1} \vy,\ \ \ \ \ \text{for any } \vx,\vy \in \R^N$$
where $\Pi = \mD_\uppi$.
%-----------------------------------
\begin{lemma}\cite{Bremaud}
The pair $(\uppi,\mQ)$ is reversible if and only if $\mQ^T$, acting as a linear operator from $\cH_{\frac{1}{\pi}}$ onto itself, satisfies
$\innprod{\mQ^T\vu}{\vecv} = \innprod{\vy}{\mQ^T\vecv}.$
\label{self adjoint}
\end{lemma} 
%-------------------------------------
Such operators are called self-adjoint or Hermitian and always have real eigenvalues. Thus, we can define the ordering of eigenvalues of $-\mQ$ as $0 = \lambda_1 < \lambda_2 \leq \lambda_3 \leq \cdots \leq \lambda_N$, just like we did for $\mL$. The eigenvectors of self-adjoint operators are orthogonal with respect to the related inner product. Hence, for any two left eigenvectors $\vecv_i$ and $\vecv_k$ of $\mQ$, $i \neq k$, we have $\innprod{\vecv_i}{\vecv_k} = 0$. Picking $i=1$, and substituting $\vecv_1 = \uppi$, we obtain
$$\innprod{\uppi}{\vecv_k} = \uppi^T \Pi^{-1} \vecv_k = \ones^T \vecv_k = 0.$$
Therefore, the Euclidean subspace that is orthogonal to $\ones$, is also a subspace of $\cH_{\frac{1}{\pi}}$ orthogonal to $\uppi$ in terms of our new inner product, $\innprod{\cdot}{\cdot}$. Also, $\vz^n(t)^T\ones = 0 \iff \innprod{\vz^n(t)}{\uppi} = 0$, giving us the desired orthogonality with $\uppi$. This allows us to recover results from Section \ref{deterministic section} by a redefinition of the Lyapunov function $V$ in terms of $\innprod{\cdot}{\cdot}$, i.e. 
$$V(\vu) = \frac{1}{2}{\innprod{\vu}{-\mQ^T\vu}}\ \ \text{for all } \vu \in \cH,$$
and appropriately using the new inner product (i.e. $\innprod{\vu}{\vw}$ instead of $\vu^T\vw$, for any two vectors $\vu$ and $\vw$) wherever necessary. The results from Section \ref{deterministic-stochastic} never required a specific form of $\mQ$ to work as long the results from Sections \ref{stochastic section} and \ref{deterministic section} held, and now therefore hold true for any choice of time reversible $\mQ$.

From the above, we have extended our main results to be applicable to any time reversible CTMC kernel $\mQ$. It is important to note that our extension is purely based on an update in the inner product used for the analysis purposes in Section \ref{deterministic section}, and requires no modification to the way our stochastic process is constructed in Section \ref{stochastic section}, making $\mQ$ simply a \emph{plug-and-play} term in our framework.

We now show the applicability of this extension to any ergodic, time reversible DTMC kernel $\mP$ as well. The Fiedler vector for such matrices is now the eigenvector corresponding to the second largest eigenvalue of $\mP$, also known as the spectral gap \cite{Levin,aldous,Bremaud}. We start by defining a CTMC  kernel $\mQ_p$ based on the time reversible $\mP$ as $\mQ_p \triangleq \mP - \eye$\footnote{See that $\mQ_p \ones = (\mP - \eye)\ones = 0$, and $\mQ_p$ has non-negative off-diagonal entries and negative diagonal entries. It is therefore a well defined CTMC kernel}. Suppose $\pi$ is the unique stationary distribution of $\mP$, such that the pair $(\pi,\mP)$ satisfies the \emph{detailed balance equation} (DBE) from Definition \ref{DBE}. Then, the pair $(\pi, \mQ_p)$ also satisfies the DBE and $\mQ_p$ is a time reversible CTMC. This means that all our theoretical results stand for $\mQ_p$ if we use it as the basis for generating our stochastic process in Section \ref{construction}. Therefore, $\vz^n(t)$ for a system of random walkers, which walk according to the CTMC $\mQ_p = \mP - \eye$ and interact in the same manner as in the previous sections, will approximate $\vecv_2$ of the DTMC $\mP$\footnote{Since eigenvectors of $\mQ_p$ and $\mP$ are the same.} in the long run.

We can also apply this extension to the random walk Laplacian $L^{rw} = \eye - \mD^{-1}\mA$, and the normalized Laplacian $\cL = \eye - \mD^{-1/2}\mA\mD^{-1/2}$, introduced earlier in Section \ref{laplacians pre}. We can do this by first observing that $\mD^{-1}\mA$ is a DTMC kernel with the degree distribution vector $\frac{1}{\vd^T\ones}\vd$ serving as its stationary distribution. The pair $\left( \frac{1}{\vd^T\ones}\vd, \mD^{-1}\mA \right)$ is also time reversible \cite{aldous, Bremaud} and satisfies the DBE. Thus, setting $\mQ_{rw} \triangleq \mD^{-1}\mA - \eye = -\mL^{rw}$ as our CTMC kernel extends all the results of the paper to the random walk Laplacian and allows us to obtain approximations for $\vecv_2^{rw}$, the Fiedler vector of the random walk Laplacian. Given a component $[\vecv_2^{rw}]_i$, the corresponding component of the Fiedler vector $\bar \vecv_2$ of the normalized Laplacian $\cL$ can be obtained by setting $[\bar \vecv_2]_i = [\vecv_2^{rw}]_i/\sqrt{d(i)}$, due to the similarity relationship between the two matrices as shown in Section \ref{laplacians pre}. Therefore, once $[\vecv_2^{rw}]_i$ is approximated, a simple localized computation involving only the degree of the concerned node allows us to obtain $[\bar \vecv_2]_i$ as well.  With this, we can now use our framework to approximate (on-the-fly) the Fiedler vectors of all the three important graph Laplacians as special cases of our extended theoretical results.

%==============================================

\section{Concluding remarks} \label{conclusion}

In this paper, given any time reversible Markov Chain kernel, we have detailed the construction of a stochastic process based on interacting random walkers.
Random walk algorithms usually relate to the leading/principal eigenvectors of their respective kernels. In fact, the field of Markov chain Monte Carlo (MCMC) is dedicated to the problem of using a random walk to sample according to a given probability distribution. Famous examples include the Metropolis Hastings Random walk and the Gibbs sampler \cite{Hastings1997,Jun_MCMC_Book,Peskun1973, LeeSIGMETRICS12}, which construct Markov chains and leverage the ergodic theorem to sample according to its first eigenvector. While usually these are restricted to sampling on undirected graphs, \cite{NMMC} samples from a directed graph using a novel method which involves \emph{mapping} a target distribution to the quasi-stationary distribution (QSD) of a sub-stochastic Markov chain. However, this technique also essentially leverages properties of a leading eigenvector, which is the QSD in this case. No such random walk type technique has been applied towards approximating the second eigenvector. By relating our stochastic process to a deterministic ODE system we show convergence to the second eigenvector, making our random walk based method a first in literature.

%==============================================

\section{Acknowledgments}
The authors would like to thank the anonymous reviewers for their constructive comments and suggestions that greatly improved the quality of this paper. This work was supported in part by National Science Foundation under Grant Nos. IIS-1910749 and CNS-1824518.
%==============================================

\bibliographystyle{ACM-Reference-Format}
\bibliography{references}

\newpage
\appendix
\section{Appendix for Section 3}
%-----------------------------------------------------------------------------------
\subsection{Proof of Proposition 3.1:}
We only show that \eqref{mean field X} implies \eqref{F_x}, since the steps relating \eqref{mean field Y} and \eqref{F_y} are exactly the same in terms of algebra.
By substituting \eqref{mean field X} into \eqref{rate x}, we obtain
\begin{equation}
F_x(\vx,\vy) 
= \sum_{i \in \cN} \sum_{j \in \cN, j \neq i}(\ve_i-\ve_j)Q_{ji}x_j    
+    \sum_{i \in \cN} \sum_{j \in \cN, j \neq i}(\ve_i-\ve_j) (\kappa x_jy_j)x_i 
= \vu + \vw,
 \label{u+w}
\end{equation}
where we use $\vu \in \R^N$ and $\vw \in \R^N$ to denote the two summation terms. Observe that the $k^\text{th}$ entry of $\vu$ can be written as
\begin{equation*}
u_k = \sum_{j \neq k} (1-0)Q_{jk}x_j + \Big[\sum_{i \neq k}(0-1) Q_{ki}\Big]x_k 
= \sum_{j \neq k} Q_{jk}x_j + Q_{kk} x_k 
= \sum_{j \in \cN} Q_{jk}x_j = [\mQ^T\vx]_k,
\end{equation*}
suggesting that $\vu = [u_k]$ can be written as $\vu = \mQ^T\vx$.
Similarly, the $k^\text{th}$ entry of $\vw$ can be written as
\begin{equation*}
\begin{split}
w_k 
&= \Big[\sum_{j \neq k} (1-0) \kappa x_jy_j\Big]x_k + \sum_{i \neq k}(0-1)(\kappa x_k y_k)x_i \\ 
&=  \Big[\sum_{j \neq k} (1-0) \kappa x_jy_j\Big]x_k + (\kappa x_k y_k)x_k 
+ \sum_{i \neq k}(0-1)(\kappa x_k y_k)x_i - (\kappa x_k y_k)x_k \\ 
&=  \Big[\sum_{j \in \cN} \kappa x_jy_j\Big]x_k - \kappa x_k y_k \sum_{i \in \cN} x_i = [\kappa \vx^T\vy ]x_k - \kappa x_k y_k,
\end{split}
\end{equation*}
suggesting that $\vw = [w_k]$ can be written as $\vw = [\kappa \vx^T \vy]\vx - \kappa \mD_\vy \vx$.
Substituting the expressions for $\vu$ and $\vw$ in \eqref{u+w}, we obtain \eqref{F_x}, which completes the proof.
\label{proof 3.1}
%------------------------------------------------------------------------------------%

%------------------------------------------------------------------------------------%
\subsection{Proof of Proposition \ref{Lipschitz prop}:}
For this proof, we make a change of notation. Vectors $(\vx,\vy) \in \R^{2N}$ will now be written as $\begin{bmatrix} \vx	\\  \vy  \end{bmatrix}.$ Similarly, we have 
$$F(\vx,\vy) = \begin{bmatrix} F_x(\vx,\vy)	\\  F_y(\vx,\vy)  \end{bmatrix} = \begin{bmatrix}  \mQ^T - \kappa \mD_\vy  &0	\\ 0  &\mQ^T - \kappa \mD_\vx   \end{bmatrix}\begin{bmatrix} \vx	\\  \vy  \end{bmatrix} + [\kappa\vx^T\vy]\begin{bmatrix} \vx	\\  \vy  \end{bmatrix}.$$

Denote by $\|\cdot \|$ the $2-$norm for any vector in $\R^{2N}$, and the induced matrix norm for any $2N \times 2N$ dimensional matrix. Using this notation for any $(\vx,\vy) \in S$ and $(\vu,\vw) \in S$, by the mean value theorem, there exists a point $(\va,\vb) \in S$ on the line segment joining $(\vx,\vy)$ and $(\vu,\vw)$ such that

\begin{equation}
\begin{split}
\|F(\vx,\vy) - F(\vu,\vw)\| = \| \J_F(\va,\vb) \| \Big\|\begin{bmatrix} \vx	\\  \vy  \end{bmatrix} - \begin{bmatrix} \vu	\\  \vw  \end{bmatrix}   \Big\|,
\end{split} 
\end{equation}
where $\J_F(\va,\vb))$ is the Jacobian matrix of $F$ evaluated at $(\va,\vb)$, given by
\begin{equation*}
\J_F(\va,\vb) =
\begin{bmatrix}
\mQ	& 0\\
0	& \mQ
\end{bmatrix}
+
\kappa \begin{bmatrix}
    \va \vb^T + \mD_\va\mD_\vb - \mD_\vb      & \va \va^T - \mD_\va \\
    \vb \vb^T - \mD_\vb       & \vb \va^T + \mD_\vb\mD_\va-\mD_\va
\end{bmatrix}.
\end{equation*}
All elements of the Jacobian matrix are bounded uniformly over the line segment joining $(\vx,\vy)$ and $(\vu,\vw)$, which in turn means that  $\| \J_F(\va,\vb) \|$ is bounded too. Therefore, $F:S\rightarrow \R^{2N}$ is Lipschitz continuous which is necessary and sufficient for the existence and uniqueness of solutions to \eqref{ODE F} \cite{ODE_Book_Hirsch, ODE_Book_Miller}.
\label{proof 3.2}
%------------------------------------------------------------------------------------%
\subsection{Proof of Theorem \ref{fluid limit}} \label{stochastic-deterministic}
\begin{proof}(Theorem \ref{fluid limit})
Before the main part of the proof, we borrow some results from literature which will be used later.
\begin{proposition} (Proposition 5.2 in \citep{Draief})
Let $Y$ be a Poisson process with unit rate $(\text{i.e. } Y(t) \sim \text{Poisson}(t)$ for all $t\geq0)$. Then for any $\epsilon > 0$ and $T>0$,
$$\mP \Big( \sup_{0 \leq t \leq T} |Y(t) - t| \geq \epsilon \Big) \leq 2\epsilon \exp{\big( -T \cdot h(\epsilon/T) \big)} $$
where $h(x) = (1+x)\log(1+x) - x$.
\label{Poisson bound}
\end{proposition}
We also state Gronwall's inequality.
\begin{lemma} (Gronwall's inequality)
Let $f$ be a bounded, real valued function on $[0,T]$ satisfying
$f(t) \leq a +b\int_0^t u(s)ds$ for all  $t \in [0,T]$,
where $a$ and $b$ are non-zero, real valued constants. Then,
$ f(t)\leq a \exp (bt)$ for all $t \in [0, T]$.
\label{Gronwall}
\end{lemma}
Observe that since $Q_{ji} = -L_{ji} = A_{ji}$ is no bigger than 1 for all $j \neq i$ and $x_i, y_i \in (0,1)$ for all $i \in \cN$,we can bound the terms in \eqref{rate x} and \eqref{rate y}, and obtain\footnote{for $\mQ\neq -\mL$, the entries are still bounded by some positive constant given by $C = \max_{i,j \in \cN} Q_{ji}$. In this case, we use $C$ instead of $1$ to bound $Q_{ji}$. The rest of the steps remain the same.}
\begin{equation}
\bar Q^n_{x:j \to i} (\vx,\vy) \leq n(1+\kappa) \ \ \text{  and  } \ \ \bar Q^n_{y:j \to i} (\vx,\vy) \leq n(1+\kappa).
\label{Q bounds}
\end{equation}
We now proceed with the main body of our proof.

For a unit rate Poisson process $Y(t)$, its centered version is given by $\hat Y(t) \triangleq Y(t) - t$ for all $t\geq0$. Let $Y_{ji}^x$ and $Y_{ji}^y$ for all $i,j\in \cN,\ j \neq i$ be independent Poisson processes of unit rate. Let $\hat{Y}_{ji}^x$ and $\hat{Y}_{ji}^y$ be their centered versions, i.e. $\hat{Y}_{ji}^x (t) = Y_{ji}^x (t) - t$ and $\hat{Y}_{ji}^y (t) = Y_{ji}^y (t) - t$ for any $t \geq 0$.

The continuous time Markov chain $\{\vx^n(t), \vy^n(t) \}_{t\geq 0}$ can be constructed for any $t \geq 0$ as

\begin{equation}
\begin{split}
\big( \vx^n(t), \vy^n(t) \big) &= \big( \vx^n(0), \vy^n(0) \big) 
+ \sum_{i \in \cN} \sum_{j \neq i} \Big(   \frac{\ve_i - \ve_j}{n}   ,  0  \Big)Y_{ji}^x \Big(  \int_0^t \bar Q^n_{x:j \to i} \big( \vx^n(s), \vy^n(s)\big) \Big) \\
&+ \sum_{i \in \cN} \sum_{j \neq i} \Big(   0  ,   \frac{\ve_i - \ve_j}{n}  \Big)Y_{ji}^y \Big(  \int_0^t \bar Q^n_{y:j \to i} \big( \vx^n(s), \vy^n(s)\big) \Big).
\end{split}
\label{non homo poisson representation}
\end{equation}
Indeed, for any $i,j \in \cN,\ j \neq i$, $\Big(\frac{\ve_i - \ve_j}{n}   ,  0 \Big)$ and $\Big(   0  ,   \frac{\ve_i - \ve_j}{n}  \Big)$ are the admissible jumps. From (\ref{rate x}) and (\ref{rate y}), at any instant $s\in[0,t]$, these jumps take place with rate $\bar Q^n_{x:j \to i} \big( \vx^n(s), \vy^n(s)\big) $ and $\bar Q^n_{y:j \to i} \big( \vx^n(s), \vy^n(s)\big) $. The Poisson processes that counts the number of such jumps up till time $t \in [0,\infty)$ are therefore non-homogeneous  Poisson processes, which are given by $Y_{ji}^x \Big( \int_0^t \bar Q^n_{x:j \to i} \big( \vx^n(s), \vy^n(s)\big) \Big)$ and $Y_{ji}^y \Big(  \int_0^t \bar Q^n_{y:j \to i} \big( \vx^n(s), \vy^n(s)\big) \Big)$.

(\ref{non homo poisson representation}) can be rewritten in terms of $F$, $\hat Y_{ji}^x$ and $\hat Y_{ji}^y$ (the centered verions of $ Y_{ji}^x$ and $ Y_{ji}^y$) as

\begin{equation}
\begin{split}
\big( \vx^n(t), \vy^n(t) \big) &= \big( \vx^n(0), \vy^n(0) \big) 
+ \sum_{i \in \cN} \sum_{j \neq i} \Big(   \frac{\ve_i - \ve_j}{n}   ,  0  \Big)\hat{Y}_{ji}^x \Big(  \int_0^t \bar Q^n_{x:j \to i} \big( \vx^n(s), \vy^n(s)\big) \Big) \\
&+ \sum_{i \in \cN} \sum_{j \neq i} \Big(   0  ,   \frac{\ve_i - \ve_j}{n}  \Big)\hat{Y}_{ji}^y \Big(  \int_0^t \bar Q^n_{y:j \to i} \big( \vx^n(s), \vy^n(s)\big) \Big) 
+ \int_0^t F\big(\vx(s),\vy(s)\big)ds.
\end{split}
\label{stochastic rep}
\end{equation}
Subtracting (\ref{Deterministic F(x,y)}) from (\ref{stochastic rep}) and taking norm yields
\begin{equation}
\begin{split}
\big\|   \big( \vx^n(t),\vy^n(t) \big) - \big( \vx(t), \vy(t) \big)  \big \| &\leq \big\|   \big( \vx^n(0), \vy^n(0) \big) - \big( \vx(0), \vy(0) \big)  \big \| \\
&+ \int_0^t \big\| F\big(\vx^n(s),\vy^n(s)\big) - F\big(\vx(s),\vy(s)\big) \big\| \\
&+\sum_{i \in \cN} \sum_{j \neq i} \Big\| \Big(   \frac{\ve_i - \ve_j}{n}   ,  0  \Big) \Big\|  \Bigg| \hat{Y}_{ji}^x \Big(  \int_0^t \bar Q^n_{x:j \to i} \big( \vx^n(s), \vy^n(s)\big) \Big) \Bigg|\\
&+\sum_{i \in \cN} \sum_{j \neq i} \Big\| \Big(   0  ,   \frac{\ve_i - \ve_j}{n}  \Big) \Big\|  \Bigg| \hat{Y}_{ji}^y \Big(  \int_0^t \bar Q^n_{y:j \to i} \big( \vx^n(s), \vy^n(s)\big) \Big) \Bigg|
\end{split}
\end{equation}
where the inequality comes from repeated applications of triangle inequality. Note that $\Big\| \Big(   0  ,   \frac{\ve_i - \ve_j}{n}  \Big) \Big\| = \Big\| \Big(  \frac{\ve_i - \ve_j}{n} , 0  \Big) \Big\| = \sqrt{2}/n$. Thus, we can rewrite the above inequality as
\begin{equation*}
\begin{split}
\big\|   \big( \vx^n(t), \vy^n(t) \big) - \big( \vx(t), \vy(t) \big)  \big \| &\leq \big\|   \big( \vx^n(0), \vy^n(0) \big) - \big( \vx(0), \vy(0) \big)  \big \| \\
&+ \int_0^t \big\| F\big(\vx^n(s),\vy^n(s)\big) - F\big(\vx(s),\vy(s)\big) \big\| \\
&+\sum_{i \in \cN}\sum_{j \neq i}\frac{\sqrt{2}}{n}   \Big| \hat{Y}_{ji}^x \Big(  \int_0^t \bar Q^n_{x:j \to i} \big( \vx^n(s), \vy^n(s)\big) \Big) \Big| \\
&+\sum_{i \in \cN}\sum_{j \neq i}\frac{\sqrt{2}}{n} \Big| \hat{Y}_{ji}^y \Big(  \int_0^t \bar Q^n_{y:j \to i} \big( \vx^n(s), \vy^n(s)\big) \Big) \Big|.
\end{split}
\end{equation*}
(A1) says $\| \big( \vx(0), \vy(0) \big) - \big( \vx^n(0), \vy^n(0) \big) \| = 0$ for all $n \in \N$. From (A1), Proposition \ref{Lipschitz prop}, and using the bounds from \eqref{Q bounds}, we obtain
\begin{equation}
\begin{split}
\big\|   \big( \vx^n(t), \vy^n(t) \big) - \big( \vx(t), \vy(t) \big)  \big \| &\leq   \int_0^t M \big\|   \big( \vx^n(t), \vy^n(t) \big) - \big( \vx(t), \vy(t) \big)  \big \| \\
&+\sum_{i \in \cN}\sum_{j \neq i}\frac{\sqrt{2}}{n} \Big[ \Big| \hat{Y}_{ji}^x \big(n(1+\kappa)t\big) \Big| +  \Big| \hat{Y}_{ji}^y \big(n(1+\kappa)t\big) \Big| \Big].
\end{split}
\label{apply gronwall}
\end{equation}
Define \begin{equation}
\varepsilon_{ji}^n(t) \triangleq \sqrt{2}\Big[ \Big| \hat{Y}_{ji}^x \big(n(1+\kappa)t\big) \Big| +  \Big| \hat{Y}_{ji}^y \big(n(1+\kappa)t\big) \Big| \Big]
\label{epsilon}
\end{equation} 
for any $i,j \in \cN,\ j \neq i$. This quantity is what we would like to control. Observe that

\begin{equation}
\begin{split}
\mP \Big( \sup_{0 \leq t \leq T} \sum_{i \in \cN}\sum_{j \neq i} \frac{1}{n} \varepsilon_{ji}^n(t) \geq \epsilon\Big)  &\leq \sum_{i \in \cN}\sum_{j \neq i} \mP \Big( \sup_{0 \leq t \leq T}  \varepsilon_{ji}^n(t) \geq \frac{n \epsilon}{N(N-1)} \Big) \\
&\leq \sum _{x,y}\sum_{i \in \cN}\sum_{j \neq i} \mP \Big( \sup_{0 \leq t \leq T}  |\hat{Y}(n(1+\kappa)t) | \geq \frac{n \epsilon}{\sqrt{2}N(N-1)} \Big) \\
&= 2N(N-1)\mP \Big( \sup_{0 \leq s \leq n(1+\kappa)T}  |\hat{Y}(s) | \geq \frac{n \epsilon}{\sqrt{2}N(N-1)} \Big) \\
&\leq 4N(N-1)\exp\Bigg( {-n(1+\kappa)T \cdot h\Big( \frac{\epsilon}{\sqrt{2}N(N-1)(1+\kappa)T} \Big)} \Bigg).
\end{split}
\end{equation}
The first inequality comes from the fact that $\mP\big( \sum_{i=1}^k X_i  \geq \epsilon\big) \leq \sum_{i=1}^k \mP \big( X_i \geq \epsilon/n \big)$.. Applying the same to $\varepsilon_{ji}^n(t)$ in \eqref{epsilon}, which contains two terms (one corresponding to \emph{x} and the other to \emph{y}), gives us the second inequality. The centered Poisson processes $\hat{Y}$ are written in an un-indexed manner to emphasize their independence, which gives us the (third) equality. Finally, last inequality is a result of applying Proposition \ref{Poisson bound}.

Applying Lemma \ref{Gronwall} (Gronwall's inequality) to (\ref{apply gronwall}), we get
\begin{equation*}
\begin{split}
\mP \Big( \sup_{0 \leq t \leq T} \big\|   \big( \vx^n(t), \vy^n(t) \big) &- \big( \vx(t), \vy(t) \big)  \big \| \geq \epsilon e^{MT} \Big) \\
&\leq \mP \Big( \sup_{0 \leq t \leq T} \big(\sum_{i \in \cN}\sum_{j \neq i} \frac{1}{n} \varepsilon_{ji}^n(t) \big)e^{MT} \geq \epsilon e^{MT}  \Big) \\
&= \mP \Big( \sup_{0 \leq t \leq T}   \sum_{i \in \cN}\sum_{j \neq i} \frac{1}{n} \varepsilon_{ji}^n(t) \geq \epsilon  \Big) \\
&\leq 2N(N-1)\exp\Bigg( {-n(1+\kappa)T \cdot h\Big( \frac{\epsilon}{2\sqrt{2}N(N-1)(1+\kappa)T} \Big)} \Bigg).
\end{split}
\end{equation*}
This can also be written as
\begin{equation*}
\begin{split}
\mP &\Big(\sup_{0 \leq t \leq T} \big\|  \big( \vx^n(t), \vy^n(t) \big) - \big( \vx(t), \vy(t) \big)  \big \| \geq \epsilon \Big) \\
&\leq 4N(N-1)\exp\Bigg({-n(1+\kappa)T \cdot h\Big( \frac{\epsilon e^{-MT}}{\sqrt{2} N (N-1) (1+\kappa) T} \Big)}\Bigg)
\end{split}
\end{equation*}
which is (\ref{Prob bound}). Now observe that the above bound is finite and decreasing exponentially as $n$ increases. Therefore, 
$$ \sum_{n=1}^\infty  \mP\Big( \sup_{0 \leq t \leq T} \big\|   \big( \vx^n(t), \vy^n(t) \big) - \big( \vx(t), \vy(t) \big)  \big \| \geq \epsilon \Big) < \infty$$
and the almost sure convergence in (\ref{convergence a.s.}) follows from the \textit{Borel Cantelli Lemma}. This completes the proof of Theorem \ref{fluid limit}.
\end{proof}
%==========================================================================================%
%==========================================================================================%
%==========================================================================================%
\section{Lyapunov theory and LaSalle invariance principle} \label{Lyapunov theory}
%=======================================================================
%=======================================================================
In this section, we collect some definitions and results from the Lyapunov theory for non-linear systems. By the term \textit{Flow} (or \textit{Semi-flow}) with respect to some deterministic system, denoted by a function $\Phi:\R \times X \rightarrow X$, we mean that $\Phi(x_0,t)$ gives the solution to that deterministic system at time $t$ and starting at $x_0$ at time $0$ (note that $t\geq 0$ always in case of semi-flows, while $t \in (-\infty,+\infty)$ for flows).

Often in literature, the term `Lyapunov function' $V:X \rightarrow \R$ is defined as a non-negative function, with $V(0) =0$ and $d(V)/dt  < 0$ for any $x \in X, x\neq 0$. This is useful for analyzing systems with only one fixed point that can be easily translated to the origin $0 \in X$ without loss of generality. However, our ODE system \eqref{ODE Z} consists of multiple fixed points. The results we state below are accordingly adjusted to cover such general cases.

\begin{definition}(Section 3 in \cite{JB_semiflow}) Consider a flow $\Phi:\R \times X \rightarrow X$ contained within some set $X$. A point in $z \in X$ is a fixed point of the flow $\Phi$ if $\Phi(t,z) = z$ for all time $t \geq 0$. A function $V:X\rightarrow\R$ is called a `Lyapunov function' if
\begin{itemize}
\item[(i)]
$V$ is continuous over $X$,
\item[(ii)]
$V(\phi(t,x)) \leq V(x)$ for all $x \in X$ and $t \geq 0$ (negative semi-definiteness),
\item[(iii)]
If $V(\psi(t)) = c$, where $c$ is some constant, for some periodic oribit $\psi(t)$ for all $t \in R$, then $\psi(t)$ is actually some fixed point $x$ of the flow $\phi$.
\end{itemize}
\label{Def Lyapunov}
\end{definition}
Such Lyapunov functions can be used to prove convergence of a flow induced by a system of ODEs using the following famous theorem.

\begin{theorem}[LaSalle invariance principle](Theorem 3.1 in \cite{JB_semiflow}, Chapter 5 in \cite{Vidyasagar}, Chapter 3 in \cite{Slotine}). Let $V:X\rightarrow\R$ be a `Lyapunov function' for some set $X$ and flow $\Phi$. Let $\gamma^+(x)$ denote the forward (in time) orbit of the flow $\Phi(\cdot,x)$, i.e. starting at $x$. If $\gamma^+(x)$ is relatively compact, and all fixed points of $\Phi$ are isolated, then for all $x \in X$, $\Phi(t,x)\rightarrow z$ for some fixed point $z$.
\label{LaSalle}
\end{theorem}

To help characterize fixed points of the system, we give the definitons of Lyapunov stability, and then state a theorem that helps characterize the stable and unstable fixed points in terms of the Lyapunov function $V:X\rightarrow \R$.
%-----------------------
\begin{definition} (Stable, unstable and asymptotically stable fixed points). 
\begin{itemize}
\item[(i)]
A fixed point $z \in X$ of a flow $\Phi$ is `stable' if for all $\epsilon > 0$, there exists a $\delta > 0$ such that for any $x \in X$, $\| x - z \| < \delta \implies \| \Phi(t,x) - z \| < \epsilon$ for all $t \geq 0$.
\item[(ii)]
A fixed point $z \in X$ of a flow $\Phi$ is `unstable' if it is not stable.
\item[(iii)]
A fixed point $z \in X$ of a flow $\Phi$ is `asymptotically stable' if it is stable and there exists a $\delta>0$ such that $x \in X$, $\| x - z \| < \delta \implies \| \Phi(t,x) - z \| \rightarrow 0$ as $t \rightarrow \infty$.
\end{itemize}
\label{Def stability}
\end{definition}
%-----------------------

\begin{theorem} (Theorems 4.1 and 4.2 in \cite{JB_semiflow}, Chapter 5 in \cite{Vidyasagar}, Chapter 4 in \cite{Slotine}).
Let $z \in X$ be an isolated fixed point of flow $\Phi$, and let $V:X\rightarrow \R$ be a Lyapunov function. Let $\gamma^+(x)$ be relatively compact for any $x \in X$ with $\gamma^+(x)$ bounded. Then,
\begin{itemize}
\item[(i)]
$z$ is `asymptotically stable' if there exists a $\delta>0$ such that $V(x)>V(z)$ for any $x \in X$ where $\|x - z\| < \delta$, implying that $z$ is a local minimizer of $\ V$.
\item[(ii)]
$z$ is `unstable' if it is not a local minimizer of $\ V$, i.e. for any $\epsilon>0$, there exists an $x \in X$ such that $\| x-z \|< \epsilon$ and $V(x) < V(z)$.
\end{itemize}
\label{stability prop}
\end{theorem}

\end{document}